\documentclass[a4paper, 10pt, english]{article}
\usepackage[utf8]{inputenc}
\usepackage[T1]{fontenc}
\usepackage{babel}
\usepackage{graphicx}
\usepackage{stmaryrd}
\usepackage[a4paper]{geometry}
\usepackage{amsmath,amsfonts,amssymb,amsthm,epsfig,epstopdf,url,array}
\usepackage{rotating}
\usepackage[colorlinks=true,citecolor=red,linkcolor=blue,pdfpagetransition=Blinds]{hyperref}
\usepackage{cleveref}
\usepackage{nameref}
\usepackage{mathabx}
\usepackage{xcolor}
\usepackage{float}
\usepackage{enumitem}
\usepackage{comment}
\Crefname{paragraph}{Section}{Sections}
\usepackage{fancyhdr}

\usepackage{fullpage} 

\numberwithin{equation}{section}

\renewcommand{\leq}{\leqslant}
\renewcommand{\geq}{\geqslant}

\newcommand{\supp}{\mathrm{supp}}

\newcommand{\dive}[1]{\mathrm{div}}

\newcommand\restr[2]{{
		\left.\kern-\nulldelimiterspace 
		#1 
		\vphantom{\big|} 
		\right|_{#2} 
}}

\newcommand{\be}{\begin{equation}}
\newcommand{\ee}{\end{equation}}
\newcommand{\ba}{\begin{eqnarray}}
\newcommand{\ea}{\end{eqnarray}}

\theoremstyle{theorem} 
\newtheorem{theorem}{Theorem}[section]
\newtheorem{proposition}[theorem]{Proposition}
\newtheorem{lemma}[theorem]{Lemma}

\newtheorem{claim}[theorem]{Claim}
\newtheorem{definition}[theorem]{Definition}
\newtheorem{obs}{Remark}[section]
 
\theoremstyle{definition}

\makeatletter
\let\original@addcontentsline\addcontentsline
\newcommand{\dummy@addcontentsline}[3]{}
\newcommand{\DeactivateToc}{\let\addcontentsline\dummy@addcontentsline}
\newcommand{\ActivateToc}{\let\addcontentsline\original@addcontentsline}
\makeatother

\begin{document}

\title{\bf Optimal cost for the null controllability of the Stokes system with controls having $n-1$ components and applications}
\author{Felipe W. Chaves-Silva\thanks{Department of Mathematics, Federal University of Para\'iba, UFPB, CEP 58050-085, Jo\~ao Pessoa-PB, Brazil.
E-mail: {\tt fchaves@mat.ufpb.br}},\, Marcos G. Ferreira-Silva\thanks{Department of Mathematics, Federal University of Para\'iba, UFPB, CEP 58050-085, Jo\~ao Pessoa-PB, Brazil.
E-mail: {\tt marcos.gabriel@academico.ufpb.br}}\, and Diego A. Souza\thanks{Dep.\ EDAN and IMUS, Univ.\ of Sevilla, Aptdo.~1160, 41080 Sevilla, Spain. E-mail: {\tt desouza@us.es}}}

\maketitle
\begin{abstract}
In this work, we investigate the optimal cost of null controllability for the $n$-dimensional Stokes system when the control acts on  $n-1$ scalar components. We establish a novel spectral estimate  for low frequencies of the Stokes operator, involving solely  $n-1$ components,  and use it to show that the cost of controllability with controls having $n-1$ components remains of the same order in time as in the case of controls with $n$ components, namely $\mathcal{O}(e^{C/T})$, i.e. the cost of null controllability is not affected by the absence of one component of the control. We also give several applications of our results.
\end{abstract}

\tableofcontents

\section{Introduction and main results}

The Stokes system, introduced by G. G. Stokes in 1845 in \cite{G.G.Stokes}, describe the flow of viscous incompressible fluids at low Reynolds numbers, where inertia (represented by the nonlinear term) is negligible compared to viscous forces. Essentially,  can be seen as a linearized form of the Navier-Stokes system, giving a simpler physically model for fluid motion.
On the other hand, its linear nature makes it a natural starting point for studying controllability of fluids, offering valuable insights for tackling the more complex systems, like the nonlinear Navier-Stokes system.

Let $ \Omega \subset \mathbb{R}^{n} $, with $ n = 2, 3 $, be a bounded domain with smooth boundary $ \partial\Omega $. Let $T > 0$, and let $\omega $ be a non-empty open subset of $ \Omega $, called \textit{control domain}. Throughout the paper, we adopt the notation $ Q := \Omega\times(0,T) $, $ \Sigma := \partial\Omega\times (0,T) $, and denote by $ \nu(x)$ the outward unit normal vector at a point $x\in \partial\Omega $. 

We begin by introducing the standard functional spaces commonly used in the context of fluid mechanics:
	$$
		L^2_{div}: = \bigl\{u\in L^{2}(\Omega)^{n}:\ \nabla\cdot u \in L^{2}(\Omega)\bigr\}, 
	$$
	$$
		H: = \bigl\{u\in L^2_{div}:\ \nabla\cdot u = 0\ \hbox{in}\ \Omega,\ \ u\cdot\nu = 0\ \hbox{on}\ \partial\Omega\bigr\},
	$$
and
	$$
		V := H\cap H_{0}^{1}(\Omega)^{n}.
	$$

In this work, we are concerned with the controlled:
	\begin{equation}\label{stk.eq}
		\left\{
		\begin{array}{lll}
			u_{t} - \Delta u + \nabla p = f1_{\omega} & \text{in}& Q,\\	
			\nabla\cdot u = 0 & \text{in}& Q,\\
			u = 0 & \text{on}& \Sigma,\\
			u(\cdot,0) = u_{0} & \text{in}& \Omega,
		\end{array}
		\right.
	\end{equation}
where $ u = u(x,t) $ denotes the velocity field, $p = p(x,t)$ is the pressure, $u_{0}$ is the initial profile of the velocity field, and $f=(f_1,f_2,\ldots, f_n)$ is a source (or control) acting in $\omega$ during the time interval $(0,T)$. System \eqref{stk.eq} is a linearized version of the classical Navier-Stokes system, which describes the time evolution of an incompressible viscous fluid within $\Omega$.

\begin{definition}
Let $T>0$. System \eqref{stk.eq} is said to be \textit{null controllable at time $T$} if, for each initial condition $u_{0} \in H$, there exists a control $f \in L^{2}(\omega\times(0,T))^n$ such that the corresponding solution\break $ u\in L^{2}((0,T);V) \cap C^{0}([0,T];H) $ of system \eqref{stk.eq} satisfies
	\begin{equation*}\label{null-cont-prop}
		u(\cdot, T) = 0 \quad \mbox{in}\quad \Omega.
	\end{equation*}
\end{definition}

The null controllability problem for the system \eqref{stk.eq} has been extensively studied in the literature. In particular, this result was obtained by O. Yu. Imanuvilov in \cite{Imanuvilov} for the case where the control region $\omega$ is an open subset of the domain. See also \cite{Cara-Guerrero-Imanuvilov-Puel-0} for a related result, and \cite{Souza-CS-Zhang} for the case where $\omega$ is a measurable set of positive measure.  

The results proved in \cite{Cara-Guerrero-Imanuvilov-Puel-0}  and \cite{Imanuvilov} and  can be stated as follows:
\begin{theorem}
Let $T>0$, and let $\omega$ be a non-empty open subset of $\Omega$. Then system \eqref{stk.eq} is null controllable at time $T$. Moreover, there exists a constant $C_{obs}(T)>0$ such that the control satisfies the  estimate 
	\begin{equation}\label{costcontrol0}
		\|f\|_{L^2(\omega \times (0,T))^{n}} \leq C_{obs}(T) \|u_0\|_{H}.
	\end{equation}
\end{theorem}

The constant $C_{obs}(T)$ in \eqref{costcontrol0}, known as the \textit{cost of controllability at time $T$}, quantifies the  effort required to drive the solution to rest at time $T$. Understanding the behavior of this constant as\break $T \to 0^{+}$ is a topic of interest in Control Theory and has attracted considerable attention in recent years. Among several important contributions, we highlight the works of E. Fernández-Cara and E. Zuazua \cite{Cara-Zuazua}, who investigated the cost of approximate controllability for the heat equation with Dirichlet boundary conditions, and T. Seidman \cite{Seidman}, who provided sharp lower bounds on the cost with respect to time for the heat equation. Later, L. Miller, improved Seidman's results by employing the spectral approach developed by G. Lebeau and L. Robbiano \cite{Lebeau-Robbiano}. We also mention the contributions \cite{BP20, Remi-Ludovick, Remi-Takeo, Lebeau-Felipe}, which address similar questions in the context of fluid mechanics.

For the Stokes system, upper bounds for the cost of controllability can be deduced from the arguments in \cite{Imanuvilov} and \cite{Cara-Guerrero-Imanuvilov-Puel-0}, yielding the estimates
	\begin{equation*}
		C_{obs}(T) \leq C_{1}e^{C_{2}/T^{8}} \quad \mbox{and} \quad C_{obs}(T) \leq C_{1}e^{C_{2}/T^{4}},
	\end{equation*}
 for some positive constants $C_1$ and $C_2$. Although these bounds are not explicitly stated in the respective references, they follow from the proofs, which rely on global Carleman inequalities.

The optimal cost,  with respect to time,  for the Stokes system was later obtained by G. Lebeau and F. W. Chaves-Silva in \cite{Lebeau-Felipe}. Using spectral methods, they proved the sharp estimate
	\begin{equation*}
		C_{obs}(T) \leq C_{1}e^{C_{2}/T},
	\end{equation*}
which matches the optimal cost of null controllability for the heat equation.

It is important to note that the works \cite{Lebeau-Felipe}, \cite{Cara-Guerrero-Imanuvilov-Puel-0}, and \cite{Imanuvilov} consider the case in which the control acts on all components. That is, the control $f$ in \eqref{stk.eq} has $n$ scalar components. This naturally leads to the following question:

\null 

\textit{
Is it possible to drive the solution of system \eqref{stk.eq} to zero using a control with fewer components?
} 

\null 

This question was first investigated by E. Fernández-Cara et al. in \cite{Cara-Guerrero-Imanuvilov-Puel}, where they showed the null controllability of system \eqref{stk.eq} using a control with only $n-1$ components. Their result holds under a geometric condition on the control region -- roughly speaking, when the control domain ``touches'' the boundary $\partial \Omega$. The proof combines the divergence-free condition with the geometric assumption and relies on global Carleman estimates with carefully constructed weight functions. Moreover, their approach yields the following upper bound for the cost of controllability:
	\begin{equation*}
		C_{obs}(T) \leq C_{1}e^{C_{2}/T^{4}}.
	\end{equation*}

Subsequently, in \cite{Coron-Guerrero}, J.-M. Coron and S. Guerrero fully solved the question posed above by removing any geometric assumptions on the control domain. Their approach also relies on global Carleman estimates, applied component-wise, combined with regularity results for the Stokes operator. More precisely, for each $ u_0 \in H $   there exists a control $ f \in L^{2}(\omega \times (0,T))^{n} $, whose $i$-th component vanishes in $Q$, that drives the solution to zero at time $T$, and with a controllability cost satisfying
	\begin{equation*}
		C_{obs}(T) \leq C_{1}e^{C_{2}/T^{9}}.
	\end{equation*}

\begin{obs}
In dimension $n = 3$, it was shown by J.-L. Lions and E. Zuazua in \cite{LionsZuazua} that, in general, it is not possible to drive the solution of system \eqref{stk.eq} to zero using a control with only one nonvanishing component (i.e., with two components identically zero).  In other words, to obtain null controllability for the Stokes one must have controls with at least $n-1$ components.
\end{obs}
	
Since the upper bounds for the controllability cost obtained in \cite{Coron-Guerrero} and \cite{Cara-Guerrero-Imanuvilov-Puel} are not of optimal, in this work we establish the optimal upper bound for the cost of controllability of system \eqref{stk.eq} using controls with  $n-1$ components.

Our main result is the following. 
\begin{theorem}\label{main.teo}
Let $T>0$ and let $ \omega $ be a non-empty open subset of $ \Omega$. For every $ u_{0} \in H $, there exists a control $ f \in L^{2}(\omega\times (0,T))^n $, 
with $f := (f_1, \ldots, f_{n-1}, 0)$, such that the associated solution $u$ of the Stokes system \eqref{stk.eq} satisfies
	\begin{equation*}
		u(\cdot,T) = 0 \quad \mbox{in}\ \Omega.
	\end{equation*}
Moreover, there exist $ C_{1}, C_{2} > 0 $, depending only on $ \Omega $ and $ \omega$ and independent of $ T $, such that the control satisfies the estimate
	\begin{equation*}\label{null.control.inq}
		\left(\sum_{k=1}^{n-1}\| f_k \|^2_{L^{2}(\omega\times(0,T))}\right)^{1/2} \leq C_{1}e^{C_{2}/T}\| u_{0} \|_{H}.
	\end{equation*}
\end{theorem}

This theorem shows that the optimal cost of control for the Stokes system \eqref{stk.eq} has the same asymptotic behavior in time, regardless of whether the control acts on all  $n$ or only $n-1$ components

\begin{obs}
In Theorem \ref{main.teo}, the assumption that the $n$-th component of the control vanishes in $Q$ is made for notational convenience. In fact, any component may be set to zero, that is, for any $i \in \{1,\dots,n\}$, one may assume that the $i$-th component vanishes while controlling the remaining $n-1$ components. This does not affect the cost estimate.
\end{obs}


To prove Theorem \ref{main.teo}, we consider  (after the change of variables $t \mapsto T-t$) the adjoint system of \eqref{stk.eq}:
\begin{equation}\label{Adjoint.stk.eq}
	\left\{
	\begin{array}{lll}
		z_{t} - \Delta z + \nabla\pi = 0  & \mbox{in} & Q,\\
		\nabla \cdot z = 0 			      & \mbox{in} & Q,\\
		z = 0                             & \mbox{on} & \Sigma,\\
		z(\cdot,0) = z_{0}                      & \mbox{in} & \Omega.
	\end{array}
	\right.
\end{equation}

It is well known that null controllability of \eqref{stk.eq} is equivalent to an appropriate observability inequality for solutions of \eqref{Adjoint.stk.eq}. Due to the absence of one component in the control, the local term in the required observability inequality involves only $n-1$ components. More precisely, we have the following result:
\begin{theorem}[Observability inequality]\label{Theorem:Observability-Inequality}
Let $T>0$ and let $ \omega $ be a non-empty open subset of $ \Omega $. There exist constants $ C_{1}, C_{2} > 0 $, depending only on $ \Omega $ and $ \omega $ and independent of $T$, such that for every initial datum $z_{0} \in H$, the solution $z = (z_{1}, z_{2}, \ldots, z_{n})$ to \eqref{Adjoint.stk.eq} satisfies
	\begin{equation}\label{adj.obs}
		\| z(\cdot,T)\|_{H} \leq C_{1}e^{C_{2}/T}\left(\int_{0}^{T}\int_{\omega}\sum_{k=1}^{n-1}|z_{k}(x,t)|^{2}dx dt\right)^{\frac{1}{2}}.
	\end{equation}
\end{theorem}

The main ingredient to prove Theorem \ref{Theorem:Observability-Inequality} is a \textit{spectral inequality} for low-frequencies of the Stokes operator involving only $n-1$ components in the observation term. The result can be stated as follows:
\begin{theorem}\label{spct-ineq-theorem}
Let $ \omega \subset \Omega $ be a non-empty open set. There exist constants $ M, K, \Lambda_1 > 0 $ such that for every sequence of complex numbers $ (a_{j})_{j\geq1} \in \ell^{2}$ and every $ \Lambda \geq \Lambda_{1} $, we have\footnote{Recall that $\ell^2\triangleq\left\{(a_j)_{j\geq1}:\,
	\sum_{j=1}^{+\infty}a_{j}^{2} < +\infty\right\} $.}
	\begin{equation}\label{spectral-inequality-1}
		\sum_{\mu_{j}\leqslant\Lambda}|a_{j}|^{2} 
		 \leq Me^{K\sqrt{\Lambda}}\int_{\omega}\sum_{k=1}^{n-1}\biggl|\sum_{\mu_{j}\leqslant \Lambda}a_{j}e_{j,k}(x)\biggr|^{2}dx,
	\end{equation} 
 where $\{e_j\}_{j\geq1}$ and $\{\mu_j\}_{j\geq1}$ denotes the eigenfunctions and eigenvalues of the Stokes operator respectively, with $e_{j} = (e_{j,1}, \ldots, e_{j,n})$. 
\end{theorem}

To prove the \textit{spectral inequality} \eqref{spectral-inequality-1}, we borrow some ideas from \cite{Lebeau-Robbiano}. To this end, we introduce the open sets:
	\begin{equation}\label{Def-sets-Z.and.W}
		\mathcal{Z} = (0,1)\times\Omega \quad \mbox{and} \quad \mathcal{W} = \left(\tfrac{1}{4},\tfrac{3}{4}\right)\times\Omega.
	\end{equation}

For $ \Lambda > 0 $, we consider the following linear combinations of the eigenfunctions (and associated eigenpressures) of the Stokes system:
	\begin{equation}\label{special-stokes-solutions}
		\Bigl(u_\Lambda(s,x),p_\Lambda(s,x)\Bigr): = \left(\sum_{\mu_{j}\leqslant\Lambda}a_{j}\frac{\sinh(s\sqrt{\mu_{j}})}{\sqrt{\mu_{j}}}e_{j}(x),\sum_{\mu_{j}\leqslant\Lambda}a_{j}\frac{\sinh(s\sqrt{\mu_{j}})}{\sqrt{\mu_{j}}}p_{j}(x)\right),
	\end{equation}
where$(s,x)\in \mathcal{Z}$, $ (a_{j})_{j}\in \ell^2 $, and $p_{j}$ denotes the pressure associated with the eigenfunctions $e_{j}$.

It is straightforward to verify that $ (u_{\Lambda},p_{\Lambda}) $ solves the augmented Stokes system
	\begin{equation}\label{elptc-system}
		\left\{
		\begin{array}{lll}
			-\partial_{ss}^{2}u_{\Lambda} - \Delta_{x}u_{\Lambda} + \nabla_{x}p_{\Lambda} = 0 & in& \mathcal{Z},\\
			\nabla_x \cdot u_{\Lambda} = 0 & in& \mathcal{Z},\\
		\end{array}
		\right.
	\end{equation}
with boundary conditions 
	\begin{equation}\label{elptc-system-boundary-conditions}
		\left\{\begin{array}{lll}
			u_{\Lambda} = 0 & on & (0,1
			)\times\partial\Omega,\\
			u_{\Lambda}(0,\cdot) = 0 & in & \Omega,\\
			\partial_{s}u_{\Lambda}(0,\cdot) = \displaystyle\sum_{\mu_{j}\leqslant\Lambda}a_{j}e_{j} & in& \Omega.
		\end{array}
		\right.
	\end{equation}

\begin{obs}
As noted in \cite{Lebeau-Felipe}, system \eqref{elptc-system} lacks a local unique continuation property even when observing all the components. Indeed, consider a function $u:\mathcal{Z}\to \mathbb{R}^n$, with $u(s,x) := b(s)\nabla_xq(x)$ such that $b \in \mathcal{C}_{0}^{\infty}(0,1) $, $b\not\equiv0$ and $\Delta_{x}q = 0$ in $\Omega$ with $\nabla_x q\not\equiv0$. Then, if $p:\mathcal{Z}\to \mathbb{R}$ is given by $p(s,x):=\partial_{ss}^{2}b(s)q(x)$, then $(u,p)$ solves locally \eqref{elptc-system}.

Since $q$ is a non-constant function, there exists $\mathcal{O} \subset\subset \Omega$ such that $\nabla_{x}q(x) \neq 0 $ for all $x \in \mathcal{O}$. On the other hand, let $s_{00} \in Supp(b)^{\complement}$ and consider $\mathrm{I}_{{s}_{00}}$, a neighborhood of $s_{00}$, such that $\mathrm{I}_{s_{00}} \subset Supp(b)^{\complement}$ and $\overline{\mathrm{I}_{s_{00}}} \cap Supp(b) \neq \emptyset$. Then, any $\mathrm{I}_{s_{00}}^{*}$ neighborhood of $s_{00}$ with $\mathrm{I}_{s_{00}} \subset\subset \mathrm{I}_{s_{00}}^{*}$ satisfies 
    $$
        \mathrm{I}_{s_{00}}^{*} \cap \ring{Supp(b)} \neq \emptyset.
    $$ 
Then, for any open sets $\mathcal{O}^{*}$ and $ \mathrm{I}_{s_{00}}^{*}$ such that $\mathcal{O}\subset \mathcal{O}^{*}$ and  $\mathrm{I}_{s_{00}} \subset\subset \mathrm{I}_{s_{00}}^{*}$,  it is easy to see that
    $$
        u\equiv 0 \quad \hbox{in}\quad \mathrm{I}_{s_{00}}\times\mathcal{O}
    $$
and  that
    $$
        u\not\equiv 0 \quad  \hbox{in} 	\quad \mathrm{I}_{s_{00}}^{*}\times\mathcal{O}^{*}.
    $$
Therefore, system \eqref{elptc-system} fails to satisfy a local unique continuation property, and consequently no local Carleman estimate can be derived for \eqref{elptc-system}.
\end{obs}

The main difficulty in working with system \eqref{elptc-system} lies in the presence of the pressure term.  A common strategy to obtain  controllability results for the Stokes system is to apply a differential operator that eliminates the pressure. For instance,   in \cite[Proposition 1]{Coron-Guerrero}, to prove a suitable observability inequality,  the authors apply the operator  $\nabla_{x}\Delta_{x}$ to $n-1$ components of the adjoint system and deduce a Carleman estimate with observation in $n-1$ components. 
\
In our case, we need a differential operator which, when applied to  $n-1$ components of $u_{\Lambda}$, allow us to recover the $L^{2}$-norm of all the components of $u_{\Lambda}$. Hence, instead of using the operator  $\nabla_{x}\Delta_{x}$, it suffices  to apply the operator $\Delta_{x}$, which proves effective because the mapping $u \mapsto \|\Delta_x u\|^{\frac{1}{2}}$ defines a Hilbertian norm on $D(-\Delta_{x})$ that is equivalent to the Sobolev $H^{2}(\Omega)$-norm.  In fact, applying $\nabla_{x}\Delta_{x}$ would not yield the norm equivalence required in our framework.

To simplify the notation, from now on, the hat notation indicates functions whose $n$-th component vanishes identically. In particular, we define $\widehat{v}_{\Lambda}$ by
	\begin{equation}\label{def:widehat(v)_Lambda}
		\widehat{v}_{\Lambda}(s,x)\ :=\ \sum_{\mu_{j}\leqslant\Lambda}a_{j}\frac{\sinh(s\sqrt{\mu_{j}})}{\sqrt{\mu_{j}}}\Delta_{x}(e_{j,1}(x),\ldots, e_{j,n-1}(x),0),
	\end{equation}
where the functions $e_{j,k}$ denote the components of the vector $e_{j} = (e_{j,1},\ldots, e_{j,n-1},e_{j,n})$.

The function $\widehat{v}_{\Lambda}$ satisfies the following system
	\begin{equation}\label{elptc-system-2}
		\left\{\begin{array}{lll}
			-\partial_{ss}^{2}\widehat{v}_{\Lambda} - \Delta_{x}\widehat{v}_{\Lambda} = 0 & \mbox{in} & \mathcal{Z},\\
			\widehat{v}_{\Lambda} = \displaystyle\sum_{\mu_{j}\leqslant\Lambda}a_{j}\frac{\sinh(\cdot\sqrt{\mu_{j}})}{\sqrt{\mu_{j}}}\Delta_{x}(e_{j,1}(x),\ldots, e_{j,n-1}(x),0)\big|_{\partial \Omega} & \mbox{on} & (0,1)\times\partial\Omega,\\
			\widehat{v}_{\Lambda}(0,\cdot) = 0 & \mbox{in} & \Omega,\\
			\partial_{s}\widehat{v}_{\Lambda}(0,\cdot) = \displaystyle \sum_{\mu_{j}\leqslant\Lambda}a_{j}\Delta_{x}(e_{j,1}(x),\ldots, e_{j,n-1}(x),0) & \mbox{in} & \Omega.
		\end{array}
		\right.
	\end{equation}

Since we do not have precise information about the boundary values of $\widehat{v}_{\Lambda}$ on $ (0,1)\times\partial\Omega $, following \cite{Lebeau-Felipe}, we introduce a small semiclassical parameter $h$, which is related to the frequency $\Lambda$:
	\begin{equation*}
		h := \frac{\delta}{\sqrt{\Lambda}}, \quad \mbox{with}\ \delta > 0\ \mbox{small}.
	\end{equation*}
This parameter concentrates the frequency support of $\widehat{v}_{\Lambda}$ near zero and allows the use of semiclassical analysis and pseudodifferential operator techniques, which are essential for the subsequent arguments.

The spectral inequality \eqref{spectral-inequality-1} will be  a consequence of the following interpolation inequality for $\widehat{v}_{\Lambda}$. 
\begin{theorem}\label{thrm.int.inq}
There exist positive constants  $\delta_0$, $\Lambda_{0}$, $ \mu$, $ C $ and $ \rho \in (0,1)$ such that: for any $ \delta\in(0,\delta_0]$  there exists  $C_{\delta}>0$ such that for all 
	$\Lambda \geqslant\Lambda_{0} $ and all sequence of complex numbers $ (a_{j})_{j\geq 1}\in \ell^2 $, the function $ \widehat{v}_{\Lambda} $, defined in \eqref{def:widehat(v)_Lambda}, satisfies 
	\begin{equation}\label{interpolation-inequality}
		\|\widehat{v}_{\Lambda}\|_{H^{1}(\mathcal{W})^{n-1}}\ \leq\ C_{\delta}\Bigl( e^{-\frac{\mu}{h}}\|\widehat{v}_{\Lambda}\|_{H^{1}(\mathcal{Z})^{n-1}} + e^{\frac{C}{h}}\|\widehat{v}_{\Lambda} \|_{H^{1}(\mathcal{Z})^{n-1}}^{1-\rho}\| \partial_{s}\widehat{v}_{\Lambda}(0,\cdot)\|_{L^{2}(\omega)^{n-1}}^{\rho}\Bigr),
	\end{equation}
where $h = \delta\Lambda^{-\frac{1}{2}} $.
\end{theorem}

The outline of this paper is as follows. 
In Section \ref{sec:preliminaries}, we introduce some definitions and preliminary results used throughout this work. 
In Section \ref{sec:proof-interp-ineq}, we prove the interpolation inequality given in Theorem \ref{thrm.int.inq}, which is the key step for obtaining the spectral inequality in Theorem \ref{spct-ineq-theorem}. The proof of this spectral inequality is the main focus of Section \ref{sec:spct-ineq}. 
Section \ref{sec:opt-null-cost} is dedicated to the proof of the observability inequality given in Theorem \ref{Theorem:Observability-Inequality}. Finally, in Section \ref{Section:Applications}, we give some applications of our results.

\section{Preliminaries}\label{sec:preliminaries}

\subsection{Stokes Operator: Definition and Properties}

In this section, we provide an introduction on the Stokes operator and its main properties.    
 
We begin by introducing  the orthogonal projector $\mathbb{P}: L^2(\Omega;\mathbb{R}^n) \mapsto H$, usually known as the \textit{Leray's Projector}.
Recall that $\mathbb{P}$ maps $H^{s}(\Omega;\mathbb{R}^{n})$ into $H^{s}(\Omega;\mathbb{R}^n)\cap H$ for all $s \geq 0$.

We denote by $\mathcal{A}$ the \textit{Stokes operator}, i.e. the self-adjoint operator on $H$ formally defined by $\mathcal{A} = -\mathbb{P}\Delta$, with domain 
	\begin{equation*}
		D(\mathcal{A}) := V\cap H^{2}(\Omega;\mathbb{R}^{n}).
	\end{equation*}
For $u\in D(\mathcal{A})$ and $f\in H$, the identity $\mathcal{A}u = f$ holds if and only if
	\begin{equation*}
		(\nabla u,\nabla v)=(f,v) \quad \forall\ v \in V.
	\end{equation*}

It is well known that $\mathcal{A} : D(\mathcal{A}) \mapsto H$ is a homeomorphism, whose inverse $\mathcal{A}^{-1}:H \mapsto H$ is self-adjoint, compact, and positive (see \cite{Book-Boyer}). Consequently, there exists a nondecreasing sequence of positive numbers $(\mu_{j})_{j=1}^{\infty}$ with $\lim_{j \to \infty}\mu_{j} = \infty $, and an associated orthonormal basis of $H$, denoted by $(e_{j})_{j=1}^{\infty}$, such that
	\begin{equation*}
		\mathcal{A} e_{j} = \mu_{j}e_{j} \quad \forall\ j \geq 1.
	\end{equation*}
In other words, the family $\{e_{j}\}_{j \geq 1}$, together with a corresponding family of pressure functions $\{p_{j}\}_{j \geq 1}$, satisfies the Stokes eigenvalue problem:
	\begin{equation}\label{eig.pro}
		\left\{\begin{array}{lll}
			-\Delta e_{j} + \nabla p_{j} =\mu_{j} e_{j} & \text{in}&\Omega,\\ 
			\nabla\cdot e_{j} = 0 & \text{in}& \Omega,\\
			e_{j} = 0 & \text{on}& \partial\Omega.
		\end{array}
		\right.
	\end{equation}

Accordingly, we can introduce the real powers of the Stokes operator. For any $r \in \mathbb{R}$, we set
	\begin{equation*}
		D(\mathcal{A}^{r})= \left\{\, u \in H :  u = \sum_{j=1}^{\infty} u_{j}e_{j}\ \hbox{ with}\ \sum_{j=1}^{\infty}\mu_{j}^{2r} |u_{j}|^{2} < \infty \,\right\}
	\end{equation*}
and
	\begin{equation*}
		\mathcal{A}^{r}u := \sum_{j=1}^{\infty}\mu_{j}^{r} u_{j}e_{j}, \quad \forall\ u = \sum_{j=1}^{\infty} u_{j}e_{j} \in D(\mathcal{A}^{r}).
	\end{equation*}

\begin{obs}
Note that the family $ \{e_{j}\}_{j\geq 1} $ forms an orthogonal basis of $ V $ and satisfies:
	\begin{equation*}
		\|\nabla e_{i}\|^{2} = \mu_{i}.
	\end{equation*} 
Moreover, 
	\begin{equation*}
		\left(\mathcal{A}^{\frac{r}{2}} e_{i},\mathcal{A}^{\frac{r}{2}} e_{j}\right) = \mu_{i}^r\delta_{ij}.
	\end{equation*} 					
\end{obs}

We conclude this section by presenting a result concerning the domains of the powers of the Stokes operator.

\begin{theorem}\label{embb}
Let $r\in \mathbb{R}$ be given and let $ \Omega $ be a bounded domain in $\mathbb{R}^{N}$ of class $ \mathcal{C}^{1,1}$. Then,
	\begin{itemize}
		\item if $-\frac{1}{2} < r\leq 2$, we have:
		\begin{equation}\nonumber
			\begin{array}{lll}
				D(\mathcal{A}^{\frac{r}{2}}) = H^{r}(\Omega)\cap H,     & \mbox{ whenever } & -\frac{1}{2} < r < \frac{1}{2}, \\
				\\
				D(\mathcal{A}^{\frac{r}{2}}) = H_{0}^{r}(\Omega)\cap H, & \mbox{ whenever } & \frac{1}{2} \leq r < 1,\\
				\\
				D(\mathcal{A}^{\frac{r}{2}}) = H^{r}(\Omega)\cap V,     & \mbox{ whenever } &  1 \leq r \leq 2.
			\end{array}
		\end{equation}
		Moreover, $ u \mapsto \left( u, \mathcal{A}^{r}u\right)^{\frac{1}{2}}$ is a Hilbertian norm in $D(\mathcal{A}^{\frac{r}{2}})$, equivalent to the usual Sobolev $H^{r}(\Omega)$-norm. In other words, there exist constants $C_{\mathcal{A},1}(r),C_{\mathcal{A},2}(r) > 0$ such that
			\begin{equation*}
				 C_{\mathcal{A},1}(r) \|u\|_{H^{r}} \leq \left( u,\mathcal{A}^{r}u\right)^{1/2} \leq C_{\mathcal{A},2}(r) \|u\|_{H^{r}}, \quad \quad \forall\ u \in D(\mathcal{A}^{\frac{r}{2}}).
			\end{equation*}
				
		\item If $\Omega$ is of class $\mathcal{C}^{k,1}$ with $k \geq 2$ and $2 < r \leq k + 1$, then
			$$
				D(\mathcal{A}^{\frac{r}{2}}) \hookrightarrow H^{r}(\Omega)\cap V
			$$
		and 
			\begin{equation*}
				\displaystyle C_{\mathcal{A},1}(r) \|u\|_{H^{r}} \leq \left( u,\mathcal{A}^{r}u\right)^{\frac{1}{2}}, \quad \forall\ u \in D(\mathcal{A}^{\frac{r}{2}}),
			\end{equation*}	
		for some positive constant $C_{\mathcal{A},1}(r)>0$.
	\end{itemize}
\end{theorem}

The proof of Theorem \ref{embb} can be found in \cite{FU-MORI}. See also \cite{Book-Boyer}.

\subsection{Some Results From Semiclassical Analysis}

This section is devoted to presenting standard notations, definitions, and results from semiclassical analysis, which will play an important role in the development of our work.  For a more comprehensive overview, we refer the reader to \cite{Martinez}.

\subsubsection{Notations and Usual Definitions}
 For $x, \xi \in \mathbb{R}^{n}$, we denote by $x \cdot \xi$ the inner product on $\mathbb{R}^{n}$ given by the expression
	$$
		x \cdot \xi = x_{1}\xi_{1} + x_{2}\xi_{2} + \ldots + x_{n}\xi_{n},
	$$
and define the \textit{Japanese bracket} $\langle \xi \rangle $  as 
	$$
		\langle \xi \rangle := \left(1 + |\xi|^{2}\right)^{\frac{1}{2}}.
	$$
For a multi-index $\alpha = (\alpha_{1}, ..., \alpha_{n}) \in \mathbb{N}^{n}$, we adopt the conventions: $|\alpha| = \alpha_{1} + \cdots + \alpha_{n}$, $\xi^{\alpha} = \xi_{1}^{\alpha_{1}}\ldots\xi_{n}^{\alpha_{n}} $, $\partial_{x}^{\alpha} = \partial_{x_{1}}^{\alpha_{1}}\ldots\partial_{x_{n}}^{\alpha_{n}}$, and $D^{\alpha} = D_{1}^{\alpha_{1}} \ldots D_{n}^{\alpha_{n}}$, where $D_{k} = -i\partial_{x_{k}}$.

For smooth functions $f = f(x,\xi)$, $g = g(x,\xi)$ with $f,g \in \mathcal{C}^{\infty}(\mathbb{R}^{n}\times\mathbb{R}^{n})$, we define the \textit{Poisson bracket} of $f$ and $g$ as 
	$$
		\{f,g\} := \sum_{j=1}^{n}\left(\partial_{\xi_{j}}f\partial_{x_{j}}g - \partial_{x_{j}}f\partial_{\xi_{j}}g\right).
	$$

The \textit{Schwartz space}, denoted by $ \mathcal{S}(\mathbb{R}^{n}) $, is the set of smooth functions $ f: \mathbb{R}^{n} \to \mathbb{C} $ satisfying the following property: for all multi-indices $ \alpha \in \mathbb{N}^{n} $ and all $ k \geqslant 0 $, 
	$$
		|\partial_{x}^{\alpha}f(x)| \leqslant C_{\alpha, k}\langle x \rangle^{-k},	\quad \forall x \in \mathbb{R}^{n},
	$$
for some constant $C_{\alpha, k} > 0$.

If $f \in \mathcal{S}(\mathbb{R}^{n})$, the \textit{Fourier transform} of  $ f : \mathbb{R}^{n} \to \mathbb{C} $ is the function $ \mathcal{F}f : \mathbb{R}^{n} \to \mathbb{C} $ defined by
\begin{equation*}
	\mathcal{F}f(\xi): = \int_{\mathbb{R}^{n}}e^{-i x\cdot\xi}f(x)dx,
\end{equation*}
where $ \xi \in \mathbb{R}^{n}$.  Moreover, the Fourier inversion formula reads:
\begin{equation*}
	f(x): = \frac{1}{(2\pi)^{n}}\int_{\mathbb{R}^{n}}e^{ix\cdot\xi }\mathcal{F}f(\xi)d\xi,
\end{equation*}
for all $f \in \mathcal{S}(\mathbb{R}^{n}) $ and $ x \in \mathbb{R}^{n} $.

\subsubsection{Semiclassical symbol and $h$-pseudodifferential operators}
Let $h \in (0, h_{0}]$ be a small positive parameter, referred to as the \textit{semiclassical parameter}, with respect to which various quantities - such as differential operators, norms, and symbols - may naturally depend.

\begin{definition}
Let $m \in \mathbb{R}$ and let $h \in (0, h_{0}]$ be a small positive parameter. A smooth function $a(x, \xi; h) \in C^{\infty}(\mathbb{R}^{n}\times\mathbb{R}^{n})$ is said to be in the \textit{semiclassical symbol class} $\mathcal{S}^{m}$ if, for every pair of multi-indices $\alpha, \beta \in \mathbb{N}^{n}$, there exists a constant $C_{\alpha, \beta} > 0$ such that
	$$
	|\partial_{x}^{\alpha} \partial_{\xi}^{\beta} a(x, \xi ; h)| \leq C_{\alpha, \beta} \langle \xi \rangle^{m - |\beta|}, \quad \forall\ (x,\xi) \in \mathbb{R}^{n} \times \mathbb{R}^{n}, \quad  \forall\ h \in (0, h_{0}].
	$$
\end{definition}

\begin{definition}
Let $ a \in \mathcal{S}^{m} $ and let $ (a_{j})_{j \geq 0} $ be a sequence of symbols, with $ a_{j} \in \mathcal{S}^{m-j} $ for any $j \geq 0$. Then, we say that $a$ is \textit{asymptotically equivalent} to the formal sum $\displaystyle \sum_{j=0}^{\infty}h^{j}a_{j}$, and we write
	$$ 
		a \simeq \sum_{j=0}^{\infty} h^{j} a_{j}, 
	$$ 
if for every $k \in \mathbb{N}$,
	$$
		a - \sum_{j=0}^{k} h^{j} a_{j} \in h^{k+1} \mathcal{S}^{m-k-1}.
	$$
\end{definition}

We call \emph{principal symbol} of $a \in \mathcal{S}^{m}$, denoted by $\sigma_{m}(a)$, the equivalence class of $a$ in $\mathcal{S}^{m}/(h\mathcal{S}^{m-1})$.

The following lemma, due to Émile Borel, ensures that such asymptotic expansions always define a symbol.
\begin{lemma}
Given $ a_{j} \in \mathcal{S}^{m-j}$, $ j=0, 1, \ldots$, there exists a symbol $a \in \mathcal{S}^{m}$ such that $\displaystyle a \simeq \sum_{j=0}^{\infty} h^{j} a_{j}$.
\end{lemma}

We now define the \textit{semiclassical pseudodifferential operator} associated with a symbol  $a \in \mathcal{S}^{m}$. 
\begin{definition}[$h$\textbf{-pseudodifferential operator}]
Let $u \in \mathcal{S}(\mathbb{R}^{n})$. For $a \in \mathcal{S}^{m}$, we define the operator $ \mathrm{Op}_{h}(a) $ on the Schwartz space $ \mathcal{S}(\mathbb{R}^{n}) $ by
	\begin{equation}\label{OP_h-definition}
		\mathrm{Op}_{h}(a)u(x) = \frac{1}{(2\pi h)^{n}} \int_{\mathbb{R}^{n}} e^{i x \cdot \xi / h} a(x,\xi;h) \mathcal{F}u\bigl(\tfrac{\xi}{h}\bigr)d\xi.
	\end{equation}
\end{definition}

The operator $\mathrm{Op}_{h}(a)$ defines a family of operators depending on the parameter $h \in (0, h_0]$. For a given $h$, it acts on the Schwartz space $\mathcal{S}(\mathbb{R}^n)$. When its symbol $a \in \mathcal{S}^m$, $\mathrm{Op}_{h}(a)$ is a continuous operator on $\mathcal{S}(\mathbb{R}^n)$ and can be uniquely extended as a continuous operator to the space of tempered distributions, $\mathcal{S}'(\mathbb{R}^n)$. 

The quantization map $a \mapsto \mathrm{Op}_{h}(a)$ is injective, a property which can be shown by the identity:
	\begin{equation*}
		a(x, \xi ; h) = e^{-ix \cdot \xi/h}\mathrm{Op}_{h}(a)\left(e^{ix \cdot \xi/h}\right).
	\end{equation*}

The class of $h$-pseudodifferential operators associated  with symbols $a \in S^{m}$ will be denoted by $\mathcal{E}^{m}$, and we write $A = \mathrm{Op}_{h}(a) \in \mathcal{E}^m$. The \emph{full symbol} of the operator $A$, is the function $a(x, \xi ; h)$ appearing in its quantization formula \eqref{OP_h-definition}, and will be denoted by $\sigma(A)$. The \emph{principal symbol} of $A$ is then given by the equivalence class of $a$ in the quotient space $\mathcal{S}^m/h\mathcal{S}^{m-1}$, that is, $\sigma_m(A) = \sigma_m(a)$. That is to say that the principal symbol of the operator $A$ coincides with the principal part of its full symbol $a$.

Finally, we define:
\begin{equation*}
	\mathcal{E} := \bigcup_{m \in \mathbb{R}} \mathcal{E}^m.
\end{equation*}
The elements of $\mathcal{E}$ are called \emph{$h$-pseudodifferential operators}. If $A \in \mathcal{E}^{m'} \subset \mathcal{E}^m$ with $m' \leq m$, then $A$ is said to be of order (at most) $m$.

   
\begin{theorem}[\textbf{Composition formula}]
Given $a \in \mathcal{S}^{m_{1}}$ and $b \in \mathcal{S}^{m_{2}}$, there exists a unique $c \in \mathcal{S}^{m_{1}+m_{2}}$ such that 
	$$
		\mathrm{Op}_{h}(a)\circ\mathrm{Op}_{h}(b) = \mathrm{Op}_{h}(c),
	$$
and we write
	$$
		c = a \# b \simeq \sum_{\alpha} \frac{h^{|\alpha|}}{i^{|\alpha|}\alpha!}\partial_{\xi}^{\alpha}a\partial_{x}^{\alpha}b.
	$$
\end{theorem}

Notice that, for every $ A \in \mathcal{E}^{m_{1}}$ and $ B \in \mathcal{E}^{m_{2}}$, we have $ AB \in \mathcal{E}^{m_{1} + m_{2}}$. Also, if
	$$
		[A,B] = AB - BA,
	$$
is the commutator between $A$ and $B$, we have 
	$$
		\frac{i}{h}[A,B] \in \mathcal{E}^{m_{1} + m_{2} - 1}.
	$$

Finally, the following formulas hold
	$$
		\sigma_{m_{1} + m_{2}}(AB) = \sigma_{m_{1}}(A)\sigma_{m_{2}}(B),
	$$
and 
	$$
		\sigma_{m_{1} + m_{2} - 1}\left(\tfrac{i}{h}[A,B]\right) = \{\sigma_{m_{1}}(A),\sigma_{m_{2}}(B)\}.
	$$

\subsection*{Some Semiclassical Carleman Estimates}\label{sec:car_boundary}
In this section, we present two local \textit{semiclassical Carleman estimates} that play a key role in proving the interpolation inequality stated in Theorem \ref{thrm.int.inq}. These estimates are necessary because, in the proof of estimate \eqref{interpolation-inequality},  we work locally and treat differently the following three cases:
\begin{itemize}
	\item[$\mathrm{(1)}$] $(s_{0}, x_{0}) \in (0,1)\times\omega$, where $s_{0} > 0$ is small;
	
	\item[$\mathrm{(2)}$] $(s_{0}, x_{0}) \in (0,1)\times\Omega$;
	
	\item[$\mathrm{(3)}$] $(s_{0}, x_{0}) \in (0,1)\times\partial\Omega$.
\end{itemize}

Note that cases $(\mathrm{1})$ and $(\mathrm{2})$ essentially correspond to interior points of the form $(s_{0}, x_{0}) \in (0,1) \times \Omega$, while case $(\mathrm{3})$ concerns boundary points $(s_{0}, x_{0}) \in (0,1) \times \partial\Omega$. 

Before presenting the \textit{semiclassical Carleman estimates} we need, we recall Hörmander’s subellipticity condition.
\begin{definition}[\textbf{Hörmander's subellipticity condition}]\label{hormander-condition}
Let $\mathcal{V}$ be an open set of $\mathbb{R}^{n+1}$, let $\varphi \in \mathcal{C}^{\infty}(\mathbb{R}^{n+1};\mathbb{R})$, and let $\mathrm{P}$ be a differential operator. We say that the pair $(\varphi,P)$ satisfies the Hörmander’s subellipticity condition in $\overline{\mathcal{V}}$ if:
    \begin{itemize}
        \item[$\mathrm{i)}$]  There is a constant $C_{0} > 0$ such that 
        							$$
        								\inf_{(s,x) \in  \overline{\mathcal{V}}} |\nabla \varphi(s,x)| > C_{0};
        							$$

        \item[$\mathrm{ii)}$] Let  $p_{\varphi}$ denotes the principal symbol of the conjugated operator
        \begin{equation*}
        	\mathrm{P}_{\varphi} := h^{2}e^{\frac{\varphi(s,x)}{h}}\mathrm{P}e^{-\frac{\varphi(s,x)}{h}}.
        \end{equation*}
 Then, there exists a constant $C_{1} > 0$ such that
            $$
                (s,x,\sigma,\xi) \in \overline{\mathcal{V}} \times \mathbb{R}^{n+1}\ \ \text{with}\ \    p_{\varphi}(s,x,\sigma,\xi) = 0 \quad \Longrightarrow \quad \frac{1}{2i}\left\{\overline{p}_{\varphi},  p_{\varphi}\right\}(s,x,\sigma,\xi) \geq C_1,
            $$
       where $\sigma \in \mathbb{R}$ and $\xi \in \mathbb{R}^n$ are the dual variables of $s$ and $x$, respectively.  
    \end{itemize}
\end{definition}

For the cases where $(s_0,x_0)$ lies in the interior, we apply the following Semiclassical Carleman estimate, whose proof can be found in \cite[Theorem 3.5]{Rousseau-Lebeau}.
\begin{theorem}[\textbf{Carleman estimate away from  the boundary}]\label{teoCarleman interior}
Let $\mathcal{V}$ be a bounded open set in $\mathbb{R}^{n+1}$ and let $(\varphi,\mathrm{P})$ satisfying the \hyperref[hormander-condition]{Hörmander's subellipticity condition} in $\overline{\mathcal{V}}$. Then, there exist constants $h_{1} > 0$ and $C > 0$ such that
	\begin{equation*}
		h\|e^{\frac{\varphi}{h}}u\|_{L^{2}}^{2} + h^{3}\|e^{\frac{\varphi}{h}}\nabla u\|_{L^{2}}^{2} \leq Ch^{4}\|e^{\frac{\varphi}{h}}\mathrm{P}u\|_{L^{2}}^{2}
	\end{equation*}
for all $u \in \mathcal{C}_c^{\infty}(\overline{\mathcal{V}})$ and $h \in (0,h_{1})$.
\end{theorem}

An interpolation estimate in the interior case, i.e., for cases $\mathrm{(1)}$ and $\mathrm{(2)}$, follows from Theorem \ref{teoCarleman interior} and is classical. In fact, Theorem \ref{teoCarleman interior} can be  used to prove the following result.
\begin{proposition}\label{Prop-interpolation-Leb-Fel}
Let $r>0$ and define
	\begin{equation}\label{def-Wr}
		\mathcal{W}_{r} := \left(\tfrac{1}{9}, \tfrac{9}{10}\right) \times \left\{x \in \Omega :\ \mathrm{dist}(x, \partial\Omega) > \tfrac{r}{2}\right\}.
	\end{equation}
Then, there exist $C>0$ and $\rho \in (0,1)$ such that for all $u \in H^{2}(\mathcal{Z})$ such that $u|_{s=0} = 0$, one has
	\begin{equation*}
		\|u\|_{H^{1}(\mathcal{W}_{r})} \leq C\|u\|_{H^{1}(\mathcal{Z})}^{1-\rho}\Bigl(\|(\partial_{ss}^2 + \Delta_x)u\|_{L^{2}(\mathcal{Z})} + \|\partial_{s}u(0,\cdot)\|_{L^{2}(\omega)}\Bigr)^{\rho}.
	\end{equation*}
\end{proposition}

For a proof of Proposition \ref{Prop-interpolation-Leb-Fel}, we refer the reader to  see \cite{Lebeau-Felipe}.

\null

In contrast to the interior case, the boundary case, i.e. case $\mathrm{(3)}$ above, requires a local geometric description of the domain near the boundary $\partial\Omega$. To address this situation, we introduce a geodesic normal coordinate system as follows: for $r>0$ small, we define the compact set
\begin{equation*}
	\mathcal{K}_{r} := \{x \in \overline{\Omega} : \text{dist}(x, \partial\Omega) \leqslant r\}.
\end{equation*}
Then, for some $r_0 \in (0,r)$, the mapping
	$$ 
		x \mapsto \left(y(x),r(x)\right),
	$$ 
where $y(x) \in \partial\Omega$ satisfies $|x - y(x)| = \text{dist}(x, \partial\Omega)$ and $r(x) := \text{dist}(x, \partial\Omega)$, is a $\mathcal{C}^\infty$-diffeomorphism from $\mathcal{K}_{r_0}$ onto $\partial\Omega \times [0, r_0]$. 

Consequently, we may locally write $x = (y, r)$, where $y \in \mathbb{R}^{n-1}$ is a local coordinate system on $\partial\Omega$, with $\partial\Omega$ corresponding to the set $\{r = 0\}$. In these coordinates, the domain $\Omega$ is locally represented by $\{r > 0\}$ and we can write $\Omega = \mathcal{V}^{'}\times(0,\varsigma)$, where $\mathcal{V}' \subset \mathbb{R}^{n-1}$ and $\varsigma > 0$ define a tubular neighborhood of $x_0$.

We now present a semiclassical Carleman estimate for the boundary case (see \cite[Theorem A.5]{Lebeau-Felipe}).
\begin{theorem}[\textbf{Carleman estimate near the boundary}]\label{carleman.estimate} 
Let $0<a < b$ and let $\mathcal{V} = (a,b)\times  \mathcal{V}^{'} \times (-\varsigma,\varsigma)$ be a neighborhood of $(s_0,x_0)$ in $\mathbb{R}\times\mathbb{R}^{n}$, where  $\mathcal{V}^{'} \times (-\varsigma,\varsigma)$ is a neighborhood of $x_0 \in \mathbb{R}^{n} $, and let $\varphi$ be a function satisfying \hyperref[hormander-condition]{Hörmander's subellipticity condition} in $\overline{\mathcal{V}}$ such that
	$$  
	\partial_{r}\varphi(s, y,r) \neq 0, \quad \forall\ (s, y, r) \in \overline{\mathcal{V}}.
	$$
Let $[\alpha,\beta] \times K'$ be a compact subset of the open set  $ (a,b) \times \mathcal{V}^{'} $ and $\varsigma'<\varsigma$. There exist constants $h_{1} > 0$ and $C > 0$ such that, for every $h \in (0,h_{1})$ and every function $u \in \mathcal{C}^{\infty}((a,b)  \times \mathcal{V}^{'}\times[0,\varsigma))$, supported in $[\alpha,\beta] \times K'\times[0,\varsigma']$, the following estimate holds:
	\begin{equation*}
		\begin{split}
			\int_{\mathcal{V}}e^{\frac{2\varphi}{h}}\left(h|u|^{2} + h^{3}|\nabla u|^{2}\right)dsdx  \leqslant C\biggl(& h^{4} \int_{\mathcal{V}}e^{\frac{2\varphi}{h}}|\mathrm{P}u|^{2}dsdx \\
			& + h\int_{\mathcal{V}^{'}}\int_{a}^{b}e^{\frac{2\varphi(s,y,0)}{h}}|u(s,y,0)|^{2}dsdy\\ 
			& + h^{3}\int_{\mathcal{V}^{'}}\int_{a}^{b}e^{\frac{2\varphi(s,y,0)}{h}}|\nabla_{y}u(s,y,0)|^{2}dsdy\\ 
			& + h^{3}\int_{\mathcal{V}^{'}}\int_{a}^{b}e^{\frac{2\varphi(s,y,0)}{h}}|\partial_{r}u(s,y,0)|^{2}dsdy \biggr).
		\end{split}
	\end{equation*}
Moreover, if $\partial_{r}\varphi(s,y,r) > 0$ for every $(s, y,r) \in \overline{\mathcal{V}}$, then the following improved estimate holds:
	\begin{equation*}
		\begin{split}
			\int_{\mathcal{V}}e^{\frac{2\varphi}{h}}\left(h|u|^{2} + h^{3}|\nabla u|^{2}\right)dsdx \leqslant C\biggl(& h^{4} \int_{\mathcal{V}}e^{\frac{2\varphi}{h}}|\mathrm{P}u|^{2}dsdx\\
			& + h\int_{\mathcal{V}^{'}}\int_{a}^{b}e^{\frac{2\varphi(s,y,0)}{h}}|u(s,y,0)|^{2}dsdy\\ 
			& + h^{3}\int_{\mathcal{V}^{'}}\int_{a}^{b}e^{\frac{2\varphi(s,y,0)}{h}}|\nabla_{y}u(s,y,0)|^{2}dsdy \biggr).
		\end{split}
	\end{equation*}
\end{theorem}

To apply Theorem \eqref{carleman.estimate} in order to handle case (3) above, we need  an appropriate weight satisfying Hörmander’s subellipticity condition near the boundary $\partial\Omega$. Let us  now construct such a weight function.

For $s_{0} \in \left[\tfrac{1}{4},\tfrac{3}{4}\right]$, consider the smooth function $ \psi : [0,1]\times\mathcal{K}_{r_{0}} \to \mathbb{R}$ given by 
\begin{equation}\label{weigh-function-psi(s,x)}
	\psi(s,x) = r(x) - (s - s_{0})^{2}.
\end{equation}
For a constant $D > 0$ large enough, we consider the weight function $\varphi : [0,1]\times\mathcal{K}_{r_{0}} \to \mathbb{R}$  defined by 
    \begin{equation}\label{weigh-function-varphi(s,x)}
    	\varphi(s,x) = e^{D\psi(s,x)}.
    \end{equation}

\begin{obs}
The weight function  $\psi$, defined in \eqref{weigh-function-psi(s,x)}, depends only on the distance from $x$ to the boundary $\partial\Omega$. Hence, using the coordinates $(s,r,y)$, the function $\psi$ becomes independent of $y$, and we may write $\psi(s,x) =\psi(s,r) $. Moreover, since $|s - s_{0}| \leq \tfrac{3}{4}$ for all $s \in [0,1]$, we have
	\begin{equation*}
		\min_{(s,r)\in [0,1]\times[0,r_0]} \psi(s,r)\geq -\frac{9}{16} \quad \hbox{and}\quad \varphi(s,r)\geq e^{-\frac{9}{16}D}.
	\end{equation*}
\end{obs}

Note that the interpolation inequality in \eqref{interpolation-inequality} is stated for solutions $\widehat{v}_{\Lambda}$ of the system \eqref{elptc-system-2}, where the operator $\mathrm{P}$ is given by
	$$
		\mathrm{P} := -\partial_{ss}^{2} - \Delta_{x}, \quad (s,x) \in (0,1)\times\Omega.
	$$
However, the Carleman estimates in Theorem \ref{carleman.estimate} are local and formulated in a  geodesic normal coordinate system. Therefore, it is necessary to rewrite the operator $\mathrm{P}$ accordingly. In such coordinates, the Laplace operator $-\Delta_{x}$ takes the explicit form
	\begin{equation*}
		-\Delta_{y,r} =  -\partial_{rr}^{2} - R(y,r,\partial_{y}) + A_{1}(x,\partial_{x}),
	\end{equation*}
where $A_{1}$ is a first-order differential operator, and $R$ is a second-order tangential operator of the form
	\begin{equation*}
		R(y,r,\partial_{y}) = \sum_{i,j = 1}^{n-1}a_{ij}(y,r)\partial_{y_{i}}\partial_{y_{j}}.
	\end{equation*}

Since first-order terms do not affect the validity of the Carleman estimate in Theorem \ref{carleman.estimate}, we may, without loss of generality, work with the  operator
	\begin{equation}\label{differential-operator-new-coordinates}
		\mathrm{P} := -\partial_{ss}^{2} -\partial_{rr}^{2} - R(y,r,\partial_{y}).
	\end{equation}
    
In the new coordinate system, the following result shows that, for sufficiently large $D > 0$, the function $\varphi$ defined in \eqref{weigh-function-varphi(s,x)} and the operator $\mathrm{P}$ given in \eqref{differential-operator-new-coordinates} satisfy Hörmander’s subellipticity condition.
\begin{proposition}\label{Prop:pair-Hormander-condition}
Let $\varphi$ be given by \eqref{weigh-function-varphi(s,x)}, and let $\mathrm{P}$ be the differential operator \eqref{differential-operator-new-coordinates}. Then, the pair $(\varphi,\mathrm{P})$ satisfies \hyperref[hormander-condition]{Hörmander's subellipticity condition} in $[0,1]\times\mathcal{K}_{r_{0}}$.
\end{proposition}

The proof of Proposition \ref{Prop:pair-Hormander-condition} follows from the  two claims, whose proofs are straightforward.
\begin{claim}
The function $ \varphi $, defined in \eqref{weigh-function-varphi(s,x)}, satisfies
	\begin{equation*}
		\partial_{r}\varphi(s,x) > 0  \quad \mbox{in}\ [0,1]\times\mathcal{K}_{r_{0}}.
	\end{equation*}
\end{claim}	

\begin{claim}
Let $D \gg 1$. Then, for every
	\begin{equation*}
		(s,y,r,\sigma,\eta,\varrho) \in (0,1)\times\partial\Omega \times[0,r_{0}] \times \mathbb{R}^{n+1},
	\end{equation*} 
where $(\sigma, \eta, \varrho) \in \mathbb{R}^{n+1}$ are the Fourier dual variables corresponding to $(s,y,r)$, the following holds: there exists a constant $C>0$ such that
	$$
		p_{\varphi}(s,y,r,\sigma,\eta,\varrho) = 0  \quad \implies \quad \frac{1}{2i}\left\{\overline{p}_{\varphi},p_{\varphi}\right\}(s,y,r,\sigma,\eta,\varrho) \geq C.
	$$
\end{claim}


\section{Proof of the Interpolation Inequality}\label{sec:proof-interp-ineq}
This section is devoted to the proof of Theorem \ref{thrm.int.inq}.  Our approach is inspired by the method developed in \cite{Lebeau-Felipe}. As discussed in Section \ref{sec:car_boundary}, the analysis is localized  near the boundary, namely in a neighborhood of $(s_{0},x_{0}) \in (0,1)\times\partial\Omega$. Throughout, we employ the notations and constructions introduced in Section \ref{sec:car_boundary}. Moreover, the constant $C > 0$ may vary from line to line. Functions whose last component is identically zero are denoted with a hat symbol.

\begin{proof}[\textbf{Proof of Theorem \ref{thrm.int.inq}}]
Let $r_{0}$ and $s_{0}$ be as defined in Section \ref{sec:car_boundary}, and let $\mathcal{W}_{r_{0}}$ as defined in \eqref{def-Wr}. For a sufficiently small parameter $s_{*} > 0$, consider the interval 
	\begin{equation*}
		[s_{0}-2s_{*},s_{0}+2s_{*}] \subset \left(\tfrac{1}{9}, \tfrac{9}{10}\right).
	\end{equation*}
	
Let $ \widehat{v}_{\Lambda} $ be the solution of system \eqref{elptc-system-2}, defined by \eqref{def:widehat(v)_Lambda}, which we localize near the boundary $\partial\Omega$ as follows. Consider a cut-off function of the form $ \chi(s,r) = \chi_{0}(s)\chi_{1}(r)$, where
	\begin{equation*} 
		\chi_{0} : (s_{0}-2s_{*},s_{0}+2s_{*}) \to [0,1] \quad \mbox{and} \quad	\chi_{1} : [0,r_{0}) \to [0,1],
	\end{equation*}
satisfying the properties
	\begin{equation}\label{Def:characteristic-functions}
		\left|\begin{array}{lllcc}
			\chi_{0}(s) = 1 \quad \text{for}\ |s-s_{0}|\ \leqslant \tfrac{3}{2}s_{*};\\
			\noalign{\smallskip}\displaystyle
			\chi_{0}\in \mathcal{C}_{c}^{\infty}\bigl((s_{0} - 2s_{*}, s_{0} + 2s_{*})\bigr),			
		\end{array}
		\right.
		\qquad \mbox{and} \qquad
		\left|\begin{array}{lllcc}
			\chi_{1}(r) = 1 \quad \text{for} \ r \in \left[0, \tfrac{r_{0}}{2}\right];\\
			\noalign{\smallskip}\displaystyle
			\chi_{1}\in \mathcal{C}_{c}^{\infty}\bigl([0,r_{0})\bigr),
		\end{array}
		\right.
	\end{equation} 	
where $	\supp(\chi_{1}) \subset [0,r_{0} - \varepsilon]$, for some small $ \varepsilon > 0 $. We then work with the localized function $\chi \widehat{v}_{\Lambda}$.

Next, to apply the semiclassical Carleman estimate from Theorem \ref{carleman.estimate}, we set
	$$ 
		\mathcal{V} := (0,1) \times \partial\Omega \times (0,r_0).
	$$ 
By Proposition \ref{Prop:pair-Hormander-condition}, the pair $(\varphi,\mathrm{P})$, where $\varphi$ is defined in \eqref{weigh-function-varphi(s,x)} and $\mathrm{P}$ is the operator given in \eqref{differential-operator-new-coordinates}, satisfies \hyperref[hormander-condition]{Hörmander's subellipticity condition} in $\overline{\mathcal{V}} = [0,1]\times \mathcal{K}_{r_{0}}$. Thus, we may apply Theorem \ref{carleman.estimate} to the function $\chi\widehat{v}_\Lambda$, which yields
	\begin{equation}\label{carleman-to-chi-v}
		\begin{split}
			&h\left\|e^{\tfrac{\varphi}{h}}\chi\widehat{v}_{\Lambda} \right\|_{L^{2}(\mathcal{Z})^{n-1}}^{2}\ + \ h^{3}\left\| e^{\tfrac{\varphi}{h}}\nabla_{s,y,r}(\chi\widehat{v}_{\Lambda}) \right\|_{L^{2}(\mathcal{Z})^{n-1}}^{2} \\
			&\leq  C\biggl( h^{4} \left\|e^{\frac{\varphi}{h}}\mathrm{P}(\chi\widehat{v}_{\Lambda}) \right\|_{L^{2}(\mathcal{Z})^{n-1}}^{2}+ h\left\| e^{\tfrac{\varphi}{h}}\chi \widehat{v}_{\Lambda}\right\|_{L^{2}((0,1)\times\partial\Omega)^{n-1}}^{2}\ +\ h^{3}\left\|e^{\tfrac{\varphi}{h}}\nabla_{s,y}(\chi \widehat{v}_{\Lambda})\right\|_{L^{2}((0,1)\times\partial\Omega)^{n-1}}^{2}\biggr),
		\end{split}
	\end{equation} 
for some constant $ C > 0 $ and all $ h \in (0,h_{1})$.

In order to rewrite \eqref{carleman-to-chi-v}, we set $\widehat{w} := e^{\tfrac{\varphi}{h}}\chi\widehat{v}_{\Lambda}$, that is
	\begin{equation*}
		\widehat{w}(s,y,r) = e^{\tfrac{\varphi(s,r)}{h}}\chi_{0}(s)\chi_{1}(r)\sum_{\mu_{j}\leqslant\Lambda}A_{j}(s)\Delta_{y,r}\widehat{e}_{j}(y,r),
	\end{equation*}
where  $A_{j}(s) := a_{j}\frac{\sinh(s\sqrt{\mu_{j}})}{\sqrt{\mu_{j}}}$ and $\widehat{e}_{j}(y,r) := (e_{j,1}(y,r),\ldots, e_{j,n-1}(y,r),0)$. 

We denote by $\widehat{w}_{0}$ the trace of $\widehat{w}$ on the boundary $\{r=0\}$, namely,
	\begin{equation*}
		\widehat{w}_{0}(s,y) := e^{\tfrac{\varphi_{0}(s)}{h}}\chi_{0}(s)\sum_{\mu_{j} \leqslant \Lambda}A_{j}(s)\Delta\widehat{e}_{j}\bigl|_{\partial\Omega}(y,0),
	\end{equation*}
where $\varphi_{0}(s) := \varphi(s,0)$.
	
Using \textit{semiclassical Sobolev norm of order one}, we see that  
	$$
		\|\widehat{w}\|_{H_{sc}^{1}(\mathcal{Z})^{n-1}}^{2} := \|\widehat{w}\|_{L^{2}(\mathcal{Z})^{n-1}}^{2} + \| h\nabla_{s,y,r}\widehat{w}\|_{L^{2}(\mathcal{Z})^{n-1}}^{2},
	$$
where $\nabla_{s,y,r} = (\partial_{s}, \partial_{y}, \partial_{r})$ denotes the gradient with respect to all variables in $\mathcal{Z}$. 

Similarly, the \textit{semiclassical tangential norm} is given by
	$$
		\|\widehat{w}_0\|_{H_{sc}^{1}(r=0)^{n-1}}^{2} := \|\widehat{w}_0\|_{L^{2}((0,1)\times\partial\Omega)^{n-1}}^{2} + \| h\nabla_{s,y}\widehat{w}_0\|_{L^{2}((0,1)\times\partial\Omega)^{n-1}}^{2}.
	$$

With these notations, estimate \eqref{carleman-to-chi-v} can be rewritten as
	\begin{equation}\label{carleman-with-w}
		\|\widehat{w}\|_{H_{sc}^{1}(\mathcal{Z})^{n-1}}^{2} \leq C \left(\|\widehat{w}_{0}\|_{H_{sc}^{1}(r=0)^{n-1}}^{2} + h^3\left\|e^{\tfrac{\varphi}{h}}\mathrm{P}(\chi \widehat{v}_{\Lambda})\right\|_{L^{2}(\mathcal{Z})^{n-1}}^{2}\right),
	\end{equation}
for some new constant $C > 0$.

The rest of the proof is devoted to estimating both sides of \eqref{carleman-with-w}. The argument is carried out in three steps. In the first step, we treat the boundary term $\|w_{0}\|_{H_{sc}^{1}(r=0)^{n}}^{2}$ by splitting it into low and high tangential frequencies and then estimating each part separately. In the second step, we estimate the term $\left\|e^{\tfrac{\varphi}{h}}\mathrm{P}(\chi \widehat{v}_{\Lambda})\right\|_{L^{2}(\mathcal{Z})^{n-1}}^{2}$. Finally, in the third step, we recover the full norm on the left-hand side, thereby completing the proof.

\null

\noindent \textbf{Step 1.} \textit{Estimate of the boundary term $\|\widehat{w}_{0}\|_{H_{sc}^{1}(r=0)^{n-1}}^{2}$}

To estimate the boundary term in the right-hand side of \eqref{carleman-with-w}, let $\Delta_{\partial\Omega}$ denote the Laplace operator on $\partial\Omega$ acting on vector fields, which admits a complete orthonormal system  $\{\zeta_{j}\}_{j\geq 1}$ in $L^{2}(\partial\Omega)$ consisting of eigenfunctions satisfying
	$$
		 -\Delta_{\partial\Omega}\zeta_{j} = \lambda_{j}^2\zeta_{j},
	$$
where $\{\lambda_{j}\}_{j\in \mathbb{N}}$ is the associated sequence of eigenvalues.

Let $\mathcal{B}_{y}$ denote the bounded operator acting on $L^{2}$-sections of the tangent bundle $T\partial\Omega$, given by
$$
\mathcal{B}_{y} := \phi\left(\sqrt{\mathrm{Id} - \Lambda^{-1}\Delta_{\partial\Omega}}\right),
$$
where $\phi \in \mathcal{C}^{\infty}(-2,2)$ is a smooth cut-off function such that $0 \leq \phi \leq 1$ and $\phi \equiv 1 $ on a neighborhood of $\left[-\sqrt{3},\sqrt{3}\right] $. 

For each $\displaystyle \sum_{j=1}^{\infty} b_j\zeta_{j} \in L^2(\partial \Omega)$, one has 
	\begin{equation*}
	\mathcal{B}_{y}\left( \sum_{j=1}^{\infty} b_j\zeta_{j} \right) = \sum_{j=1}^{\infty}\phi\left(\sqrt{1 + \Lambda^{-1}\lambda_{j}^{2}}\right)b_j\zeta_{j}
\end{equation*}	
and, since  $\widehat{w}_{0} \in L^{2}(\partial\Omega)^{n-1}$, we can apply $\mathcal{B}_y$ component-wise to decompose $\widehat{w}_{0}$ into its low and high tangential frequencies:
	\begin{equation}\label{w0-splitted}
		\begin{split}
			\widehat{w}_{0}(s,y) & =\ \mathcal{B}_{y}\widehat{w}_{0}(s,y) + (\mathrm{Id} - \mathcal{B}_{y})\widehat{w}_{0}(s,y)\\
				   & :=\ \widehat{w}_{1}(s,y) + \widehat{w}_{2}(s,y).
		\end{split}
	\end{equation}

Therefore, from \eqref{carleman-with-w}, we get
	\begin{equation}\label{carleman-with-w1-and-w2}
		\|\widehat{w}\|_{H_{sc}^{1}(\mathcal{Z})^{n-1}}^{2} \leqslant C \left(\|\widehat{w}_{1}\|_{H_{sc}^{1}(r=0)^{n-1}}^{2} + \|\widehat{w}_{2}\|_{H_{sc}^{1}(r=0)^{n-1}}^{2}  + h^{3}\left\| e^{\tfrac{\varphi}{h}}\mathrm{P}(\chi\widehat{v}_{\Lambda})\right\|_{L^{2}(\mathcal{Z})^{n-1}}^{2}\right).
	\end{equation}

\begin{obs}
As mentioned in \cite{Lebeau-Felipe}, the operator $\mathcal{B}_{y}$ is a semiclassical pseudodifferential operator of order $0$, with semiclassical parameter $\Lambda^{-\frac{1}{2}}$, acting only on the tangential variable $y\in\partial\Omega $. Its semiclassical principal symbol is given by 
	$$
		\sigma(\mathcal{B}_{y}) = \phi \left(\sqrt{1+|\eta|_{y}^{2}} \right)\mathrm{Id},
	$$
where $ |\eta|_{y}^{2} $ denotes the squared Riemannian length of the covector $\eta \in T_{y}^{*}\partial\Omega$. By definition of the essential support of $\mathcal{B}_{y}$, we have that  $w_{1}$ is microlocally supported in the region $|\eta|_{y} \leq \sqrt{3}$ and  $w_{2}$ is microlocally supported in  $|\eta|_{y} > \sqrt{2}$. 
\end{obs}

We estimate $\|\widehat{w}_{2}\|_{H_{sc}^{1}(r=0)^{n-1}}^{2}$ using the following result, which, to the best of our knowledge, does not appear in the literature and is of independent interest.
\begin{proposition}\label{Prop:w2-estimate}
Let $w_{2}$ be defined by
	\begin{equation*}
		w_{2}(s,y) := (\mathrm{Id} - \mathcal{B}_{y})\left(e^{\tfrac{\varphi_{0}(s)}{h}}\chi_{0}(s)\sum_{\mu_{j} \leqslant \Lambda}A_{j}(s)\Delta e_{j}\bigl|_{\partial\Omega}(y,0)\right).
	\end{equation*}
Then there exist positive constants $\epsilon_{0}$, $\Lambda_{0}$ and $C_{0}$ and $\rho \in (0,1)$ such that,    for every $N \in \mathbb{N}$ and all $\Lambda \geq \Lambda_{0}$, it  holds:
	\begin{equation*}
		\begin{split}
			\|w_{2}\|_{H_{sc}^{1}(r=0)^{n}}^{2} \leq  C\Lambda^{-2N}\Bigl(\|\widehat{w}\|_{L^{2}(\mathcal{Z})^{n-1}}^{2} &\ + e^{\frac{2\left(\varphi_0(s_{0}) - \epsilon_{0}\right)}{h}}\|\widehat{v}_{\Lambda}\|_{L^{2}(\mathcal{Z})^{n-1}}^{2}\\ &\ + e^{\frac{2C_{0}}{h}}\left\|\widehat{v}_{\Lambda}\right\|_{H^{1}(\mathcal{Z})^{n-1}}^{2(1-\rho)}\left\|\partial_{s}\widehat{v}_{\Lambda}(0,\cdot) \right\|_{L^{2}(\omega)^{n-1}}^{2\rho}\Bigr),
		\end{split}
	\end{equation*}
for some constant $C>0$ not depending on $\Lambda$.
\end{proposition}

Notice that, in Proposition \ref{Prop:w2-estimate}, we estimate the tangential high-frequencies of all components of the solution at the boundary by a right-hand side which contain only terms involving $n-1$ components of the solution in the interior. We prove it in Appendix \ref{Appendix:Proof-of-Prop-3.1}.

Thus, from \eqref{carleman-with-w1-and-w2}, it follows that there exist positive constants $C, C_0, \epsilon_{0}$ and  $\Lambda_{0},$  and $\rho \in (0,1)$, such that
	\begin{equation*}
		\begin{split}
			\|\widehat{w}\|_{H_{sc}^{1}(\mathcal{Z})^{n-1}}^{2} \leq\ & C \Bigl(\|\widehat{w}_{1}\|_{H_{sc}^{1}(r=0)^{n-1}}^{2} + h^{3}\left\| e^{\tfrac{\varphi}{h}}\mathrm{P}(\chi\widehat{v}_{\Lambda})\right\|_{L^{2}(\mathcal{Z})^{n-1}}^{2}\Bigr) \\
			 & + C\Lambda^{-2N}\Bigl( e^{\frac{2\left(\varphi_0(s_{0}) - \epsilon_{0}\right)}{h}}\|\widehat{v}_{\Lambda}\|_{L^{2}(\mathcal{Z})^{n-1}}^{2} + e^{\frac{2C_{0}}{h}}\left\|\widehat{v}_{\Lambda}\right\|_{H^{1}(\mathcal{Z})^{n-1}}^{2(1-\rho)}\left\|\partial_{s}\widehat{v}_{\Lambda}(0,\cdot) \right\|_{L^{2}(\omega)^{n-1}}^{2\rho}\Bigr),
		\end{split}
	\end{equation*}
for all $N \in \mathbb{N}$ and $\Lambda \geq \Lambda_{0}$.

To estimate the term $\|\widehat{w}_{1}\|_{H_{sc}^{1}(r=0)^{n-1}}^{2}$, we introduce the semiclassical pseudodifferential operator $\mathfrak{B}_{s}$ acting on the $s$-variable, given by 
	\begin{equation}\label{pseudo-differential-operator-in-s-var}
		\mathfrak{B}_{s}u(s) = \frac{1}{2\pi h}\int e^{\tfrac{i(s-s')\sigma}{h}}b_{1}(\sigma)b_{0}(s' - s_{0})u(s') ds' d\sigma,
	\end{equation}
where $ b_{j} \in \mathcal{C}_{0}^{\infty}\bigl((-\alpha_{j}, \alpha_{j})\bigr) $ satisfy $ b_{j} \equiv 1 $ on $ [-\tfrac{\alpha_{j}}{2},\tfrac{\alpha_{j}}{2}] $, for some small constants $ \alpha_{j} > 0 $.

Using this operator, we decompose $\widehat{w}_{1}$ as
	\begin{equation}\label{w1-pos-h.pdiff}
		\|\widehat{w}_{1}\|_{H_{sc}^{1}(r=0)^{n-1}}^{2} \leq 2\|\mathfrak{B}_{s}\widehat{w}_{1}\|_{H_{sc}^{1}(r=0)^{n-1}}^{2} + 2\|(1-\mathfrak{B}_{s})\widehat{w}_{1}\|_{H_{sc}^{1}(r=0)^{n-1}}^{2}.
	\end{equation}	

To estimate the two terms on the right-hand side of \eqref{w1-pos-h.pdiff}, we use the following results. 
\begin{claim}\label{claim-2}
There exist constants $\alpha_{j} > 0$ and $\delta_{1} > 0$ such that for all $ \delta \in (0,\delta_{1}) $ and  every $N \in \mathbb{N}$, there exists $ C_{N} > 0 $ such that satisfying 
	\begin{equation*}
		\|\mathfrak{B}_{s}\widehat{w}_{1}\|_{H_{sc}^{1}(r=0)^{n-1}}^{2} \leq C_{N}h^{N}\|\widehat{w}\|_{L^{2}(\mathcal{Z})^{n-1}}^{2}.
	\end{equation*}
\end{claim}	

Since the proof of Claim \ref{claim-2} follow the same steps in \cite[Claim 1]{Lebeau-Felipe}, we omit it. 

\begin{proposition}\label{Proposition:(1-Bs)w_1--estimate}
Let $w_{1}$ be defined as
	\begin{equation*}
		w_{1}(s,y) := e^{\tfrac{\varphi_{0}(s)}{h}}\chi_{0}(s)\sum_{\mu_{j} \leqslant \Lambda}A_{j}(s)\mathcal{B}_{y}\Bigl(\Delta e_{j}\bigl|_{\partial\Omega}\Bigr)(y,0).
	\end{equation*}
There exist positive constants $C$, $\delta_{0}$ and $c_{0}$ such that, for all $\delta \in (0,\delta_{0}]$, the following estimate holds:
	\begin{equation}\label{1-B-G1-estimate}
		\|(1 - \mathfrak{B}_{s})w_{1}\|_{H_{sc}^{1}(r=0)^{n}}^{2} \leq Ce^{\tfrac{2(\varphi_{0}(s_{0}) - c_{0})}{h}}\|\widehat{v}_{\Lambda}\|_{H^{1}(\mathcal{Z})^{n-1}}^{2}.
	\end{equation}
\end{proposition}

In Proposition \ref{Proposition:(1-Bs)w_1--estimate}, we are estimating the high-frequencies on $\sigma$ of all components of the solution at the boundary by a right-hand side which contain only $n-1$ components of the solution. We emphasize that estimate \eqref{1-B-G1-estimate} does not appear in \cite{Lebeau-Felipe}.  In fact, there a similar estimate is proved,  but it involves all components of the solution on the right-hand side. We prove Proposition \ref{Proposition:(1-Bs)w_1--estimate} in Appendix \ref{Appendix:Proof-Claim.3.2--and--Prop.3.3}.

Therefore, we have that 
	\begin{equation}\label{final-w2-contribution-3}
		\begin{split}
			\|\widehat{w}\|_{H_{sc}^{1}(\mathcal{Z})^{n-1}}^{2} & \leq\ C \Bigl(e^{\tfrac{2(\varphi_{0}(s_{0}) - c_{0})}{h}}\|\widehat{v}_{\Lambda}\|_{H^{1}(\mathcal{Z})^{n-1}}^{2} + h^{3}\left\| e^{\tfrac{\varphi}{h}}\mathrm{P}(\chi\widehat{v}_{\Lambda})\right\|_{L^{2}(\mathcal{Z})^{n-1}}^{2}\Bigr) \\
			& +\ C\Lambda^{-2N}\Bigl( e^{\frac{2\left(\varphi_0(s_{0}) - \epsilon_{0}\right)}{h}}\|\widehat{v}_{\Lambda}\|_{L^{2}(\mathcal{Z})^{n-1}}^{2} + e^{\frac{2C_{0}}{h}}\left\|\widehat{v}_{\Lambda}\right\|_{H^{1}(\mathcal{Z})^{n-1}}^{2(1-\rho)}\left\|\partial_{s}\widehat{v}_{\Lambda}(0,\cdot) \right\|_{L^{2}(\omega)^{n-1}}^{2\rho}\Bigr),
		\end{split}
	\end{equation}
for some positive constants $C_0, C, c_0, \epsilon_0, \Lambda_0$, $\rho \in (0,1)$ and all  $\Lambda \geq \Lambda_{0}$ and $N \in \mathbb{N}$. 
	
\null

\noindent \textbf{Step 2.} \textit{Estimate of the term $\left\| e^{\tfrac{\varphi}{h}}\mathrm{P}(\chi \widehat{v}_{\Lambda})\right\|_{L^{2}(\mathcal{Z})^{n-1}}^{2}$}

Since $\mathrm{P}\widehat{v}_{\Lambda} =0$, we have that
	\begin{equation*}
		\mathrm{P}(\chi\widehat{v}_{\Lambda}) = \chi_{0}\left[\mathrm{P},\chi_{1}\right]\widehat{v}_{\Lambda}  + \left[\mathrm{P},\chi_{0}\right]\chi_{1}\widehat{v}_{\Lambda},
	\end{equation*}
with the commutators $\left[\mathrm{P},\chi_{1}\right]$ and $\left[\mathrm{P},\chi_{0}\right]$ being differential operators of order  one. 

By the construction of the cut-off functions $\chi_{0}$ and $\chi_{1}$, it follows that
	\begin{equation*}
		\supp\left[\mathrm{P},\chi_{0}\right] \subset \mathcal{V}_{1}  \quad \mbox{and} \quad \supp\left[\mathrm{P},\chi_{1}\right] \subset \mathcal{V}_{2}
	\end{equation*}
where 
	\begin{equation}\nonumber\label{V1}
		\mathcal{V}_{1} := \Bigl([s_{0}-2s_{*},s_{0}-\tfrac{3s_{*}}{2}]\cup[s_{0}+\tfrac{3s_{*}}{2}, s_{0}+2s_{*}]\Bigr)\times \partial\Omega\times \left[0,r_{0}\right]
	\end{equation}
and
	\begin{equation}\nonumber\label{V2}
		\mathcal{V}_{2} := \left[s_{0}-2s_{*},s_{0}+2s_{*}\right]\times\partial\Omega\times\left[\tfrac{r_{0}}{2}, r_{0}\right].
	\end{equation}

Then, it follows that
	\begin{equation*}
		\left\| e^{\tfrac{\varphi}{h}}\mathrm{P}(\chi\widehat{v}_{\Lambda})\right\|_{L^{2}(\mathcal{Z})^{n-1}}^{2} \leq 2e^{\tfrac{2\varphi}{h}}\left\|\widehat{v}_{\Lambda}\right\|_{H^{1}(\mathcal{V}_{1})^{n-1}}^{2}  + 2e^{\tfrac{2\varphi}{h}}\left\|\widehat{v}_{\Lambda}\right\|_{H^{1}(\mathcal{V}_{2})^{n-1}}^{2}.
	\end{equation*}

Since
	$$
		\Bigl[s_{0}-2s_{*},s_{0}-\tfrac{3s_{*}}{2}\Bigr]\cup\Bigl[s_{0}+\tfrac{3s_{*}}{2}, s_{0}+2s_{*}\Bigr] \subset \left\{s : |s - s_{0}| > s_{*}\right\},
	$$
from Lemma \ref{claim-1}, there exists $\epsilon_{0} > 0$ such that
	\begin{equation*}
		e^{\tfrac{2\varphi}{h}}\left\|\widehat{v}_{\Lambda}\right\|_{H^{1}(\mathcal{V}_{1})^{n-1}}^{2} \leq e^{\tfrac{2\left(\varphi_{0}(s_{0}) - \epsilon_{0}\right)}{h}}\left\|\widehat{v}_{\Lambda}\right\|_{H^{1}(\mathcal{V}_{1})^{n-1}}^{2}.
	\end{equation*}
Note that, here, the constant $\epsilon_0$ can be taken to be the same as in Proposition \ref{Prop:w2-estimate}.

On the other hand, by applying Proposition \ref{Prop-interpolation-Leb-Fel} component-wise, there exist $C>0$ and $ \rho \in (0,1)$ for which
\begin{equation}\label{estinterV2}
	e^{\tfrac{2\varphi}{h}}\left\|\widehat{v}_{\Lambda}\right\|_{H^{1}(\mathcal{V}_{2})^{n-1}}^{2} \leq Ce^{\tfrac{2D_{1}}{h}} \|\widehat{v}_{\Lambda} \|_{H^{1}(\mathcal{Z})^{n-1}}^{2(1- \rho)}\|\partial_{s}\widehat{v}_{\Lambda}(0,\cdot)\|_{L^{2}(\omega)^{n-1}}^{2\rho},
\end{equation}
where  $D_{1} > 0$ satisfies $\varphi(s,r) \leq D_{1}$ on $\mathcal{V}_2$. Here, since $\mathcal{V}_{2}\subset \mathcal{W}_{r_0}$, the parameter  $\rho$ in \eqref{estinterV2} can be taken as the one given in Proposition \ref{Prop:w2-estimate}.

Therefore, from \eqref{final-w2-contribution-3}--\eqref{estinterV2}, it follows that
	\begin{equation}\label{final-w2-contribution-6}
		\|\widehat{w}\|_{H_{sc}^{1}(\mathcal{Z})^{n-1}}^{2}  \leq C\Biggl(e^{\tfrac{2(\varphi_{0}(s_{0}) - \mu)}{h}}\|\widehat{v}_{\Lambda}\|_{H^{1}(\mathcal{Z})^{n-1}}^{2} + e^{\tfrac{2C_{1}}{h}}\|\widehat{v}_{\Lambda} \|_{H^{1}(\mathcal{Z})^{n-1}}^{2(1- \rho)}\|\partial_{s}\widehat{v}_{\Lambda}(0,\cdot)\|_{L^{2}(\omega)^{n-1}}^{2\rho}\Biggr),
	\end{equation}
for some positive constants $C, C_1, \mu, \Lambda_0$, all $\Lambda \geq \Lambda_0$, $h = \delta\Lambda^{-\frac{1}{2}}$ and $\delta \in (0,\delta_{0}]$.

\null
	
\noindent \textbf{Step 3.} \textit{Conclusion}

We recall  that $\chi_{0}(s)\chi_{1}(r) = 1$ in the open set $\{s : |s - s_{0}| < s_{*}\}\times\{r : 0 \leq  r \leq \frac{r_{0}}{2}\}$. Also, there exists $\alpha > 0$ sufficiently small with  $(s_{0}-\alpha,s_{0}+\alpha) \subset (s_{0}-s_{*},s_{0}+s_{*})$ and 
	\begin{equation*}
		\varphi(s,r) > \varphi_{0}(s_{0}) - \frac{\mu}{2} \quad \mbox{for all}\ (s,r) \in\ \widetilde{\mathcal{U}} := (s_{0}-\alpha,s_{0}+\alpha) \times \left\{r : r < \tfrac{r_{0}}{2}\right\}.
	\end{equation*}

Since $\widetilde{\mathcal{U}}$ is an open subset of $\mathcal{Z}$, estimate \eqref{final-w2-contribution-6} and the fact that $h^{2}\|\widehat{v}_{\Lambda}\|_{H^{1}}^{2} \leq \|\widehat{v}_{\Lambda}\|_{H_{sc}^{1}}^{2}$ yields
	\begin{equation}\label{final-w2-contribution-8}
		\|\widehat{v}_{\Lambda}\|_{H^{1}(\widetilde{\mathcal{U}})^{n-1}}^{2}  \leq Ch^{-2}\Biggl(e^{-\tfrac{\mu}{h}}\|\widehat{v}_{\Lambda}\|_{H^{1}(\mathcal{Z})^{n-1}}^{2} + e^{\tfrac{2C_{1}}{h}}\|\widehat{v}_{\Lambda} \|_{H^{1}(\mathcal{Z})^{n-1}}^{2(1- \rho)}\|\partial_{s}\widehat{v}_{\Lambda}(0,\cdot)\|_{L^{2}(\omega)^{n-1}}^{2\rho}\Biggr)
	\end{equation}
and, since  the set $ \mathcal{U} = \displaystyle \left[\tfrac{1}{4},\tfrac{3}{4}\right] \times \{r : 0 \leq  r \leq \tfrac{r_{0}}{2}\}$ is compact, estimate  \eqref{final-w2-contribution-8} holds in $ \mathcal{U}$.

Thus, Theorem \ref{interpolation-inequality} follows from the fact $\mathcal{W} \subset \mathcal{U} \cup \mathcal{W}_{r_0}$  and Proposition \ref{Prop-interpolation-Leb-Fel} applied to $\mathcal{W}_{r_0}$.  
\end{proof}

\section{Proof of the Spectral Inequality}\label{sec:spct-ineq}
In this section, we use the interpolation inequality given in Theorem \ref{thrm.int.inq} to prove Theorem \ref{spct-ineq-theorem}. As we said before, this inequality plays a key role in establishing Theorem \ref{Theorem:Observability-Inequality}, which will be addressed in the next section.

Let $\widetilde{\omega}$ be a nonempty open set such that $\widetilde{\omega} \subset\subset \omega$, and let $\mathcal{Z}$ and $\mathcal{W}$ be the sets defined in \eqref{Def-sets-Z.and.W}. From Theorem \ref{thrm.int.inq}, we know that
	\begin{equation}\label{proof-des-interp}
		\|\widehat{v}_\Lambda \|_{H^{1}(\mathcal{W})^{n-1}}\ \leqslant\ C \left( e^{-\frac{\mu}{h}}\|\widehat{v}_\Lambda \|_{H^{1}(\mathcal{Z})^{n-1}}\ +\ e^{\frac{C}{h}}\|\widehat{v}_\Lambda \|_{H^{1}(\mathcal{Z})^{n-1}}^{1-\rho}\| \partial_{s}\widehat{v}_\Lambda(0,\cdot) \|_{L^{2}(\tilde{\omega})^{n-1}}^{\rho} \right),
	\end{equation}	
for some positive constants $ \delta_{0}$, $\Lambda_{0}$, $\mu$, $C $, $\rho \in (0,1)$ any  $ \delta \in (0,\delta_{0}]$ and all $ \Lambda \geqslant \Lambda_{0} $.

By noticing that
	\begin{equation*}
		\|\widehat{v}_{\Lambda} \|_{H^{1}(\mathcal{W})^{n-1}}^{2} \geq \|\Delta_{x}\widehat{u}_{\Lambda} \|_{L^{2}(\mathcal{W})^{n-1}}^{2},
	\end{equation*}
using the divergence-free condition, Poincaré's inequality, and the fact that $ \sinh(s) \geqslant s $,  it follows that
	\begin{equation}\label{below-estimate-for-hat(v).Lambda}
	C\|\widehat{v}_{\Lambda} \|_{H^{1}(\mathcal{W})^{n-1}}^{2} 
	 \geq\ \| u_{\Lambda} \|_{L^{2}(\mathcal{W})^{n}}^2 \geq \sum_{\mu_{j} \leqslant \Lambda}|a_{j}|^{2},
\end{equation}
for some  $C>0$.

Also, it is straightforward to see that 
	\begin{equation}\label{estimate-for-hat.v.Lambda.in.H1}
	\|\widehat{v}_{\Lambda}\|^2_{H^{1}(\mathcal{Z})^{n-1}} \leq\ C\Lambda^{2}e^{2\sqrt{\Lambda}}\sum_{\mu_{j}\leqslant \Lambda}|a_{j}|^{2}.
\end{equation}

Take $\delta_{1} \in (0, \delta_{0}]$ such that  $\delta_1 \leq \frac{2\mu}{3}$, and let  $\Lambda_{1}$ be so that $C\Lambda^{2}e^{-\sqrt{\Lambda}}\leq \frac{1}{2}$  for all $ \Lambda \geq \Lambda_{1} $.  Then, combining \eqref{proof-des-interp}, \eqref{below-estimate-for-hat(v).Lambda} and \eqref{estimate-for-hat.v.Lambda.in.H1}, and using that $ h = \delta\Lambda^{-\frac{1}{2}} $,  we obtain
	\begin{equation}\label{proof-des-interp-L2-05}
		\sum_{\mu_{j} \leqslant \Lambda}|a_{j}|^{2} \leq Me^{K\sqrt{\Lambda}}\|\partial_{s}\widehat{v}_{\Lambda}(0,\cdot)\|_{L^{2}(\tilde{\omega})^{n-1}}^{2}, 
	\end{equation}	
for some positive constants $M$ and $K$.
	
Consider $ \theta \in \mathcal{C}_0^{\infty}(\omega)$  with $ 0 \leq \theta \leq 1 $ and $ \theta \equiv 1 $ in $ \widetilde{\omega}$. Thus, we have
\begin{equation} \label{quasespectralprova}
		\begin{split}
		\|\partial_{s}\widehat{v}_{\Lambda}(0,\cdot)\|_{L^{2}(\tilde{\omega})^{n-1}}^{2}  & \leq \int_{\omega} \theta\Biggl|\sum_{\mu_{j}\leq \Lambda}a_{j}\Delta_{x}\widehat{e}_{j}(x)\Biggr|^{2}dx \\ 
		& \leq   \left(\int_{\omega}\Biggl|\sum_{\mu_{j}\leq \Lambda}a_{j}\widehat{e}_{j}(x)\Biggr|^{2}dx\right)^{\frac{1}{2}}\left(\int_{\omega}\Biggl|\Delta_{x}\Biggl(\theta\Delta_{x}\sum_{\mu_{j}\leq \Lambda}a_{j}\widehat{e}_{j}(x)\Biggr)\Biggr|^{2}dx\right)^{\frac{1}{2}} \\
		& \leq C\Lambda^{2}\left(\int_{\omega}\Biggl|\sum_{\mu_{j}\leq \Lambda}a_{j}\widehat{e}_{j}(x)\Biggr|^{2}dx\right)^{\frac{1}{2}}\left(\sum_{\mu_{j}\leq \Lambda}|a_{j}|^{2}\right)^{\frac{1}{2}}.
	\end{split}
\end{equation}

Finally, from \eqref{proof-des-interp-L2-05} and \eqref{quasespectralprova}, we obtain the spectral inequality \eqref{spectral-inequality-1}.

\section{Optimal Observability Inequality}\label{sec:opt-null-cost}
In this section, we prove Theorem \ref{Theorem:Observability-Inequality}. The proof follows some ideas of T. Seidman in \cite{Seidman} (for more details, see \cite{Lebeau-Felipe, LucMiller}) and is based on the obtainment of the following approximate observability estimate. 

\begin{lemma}\label{lemma2} 
Let $ T_{0} > 0 $ and $ \beta > 0 $. Suppose that, for every $ z_{0} \in H^{2}(\Omega)^{n-1} \cap V$, the corresponding solution $z = (z_1, \ldots, z_n)$ (with some pressure) of the adjoint system \eqref{Adjoint.stk.eq} satisfies the approximate observability estimate 
	\begin{equation}\nonumber
		h(T)\|z(\cdot,T)\|_{H}^{2} - g(T)\|z_{0}\|_{H}^{2} \leq \int_{0}^{T}\int_{\omega}\sum_{k=1}^{n-1}|z_{k}(x,t)|^{2}dxdt, \quad \forall T \in (0,T_0],
	\end{equation}
where $ h(T) = h_{0}e^{-2/(d_{2}T)^{\beta}} $ and $ g(T) = g_{0}e^{-2/(d_{1}T)^{\beta}} $, for some positive constants $ h_{0}, g_{0}$ with $d_{1} < d_{2} $.  
	 
Then, for every $ d \in (0, d_{2}- d_{1})$,  there exists $ T' \in (0,T_{0}] $ such that observability estimate \eqref{adj.obs} holds with the observability constant satisfying
	$$
		C_{obs}^{2} \leq \frac{e^{2/(dT)^{\beta}}}{h_{0}}, \quad  \forall T \in (0,T'].
	$$
Moreover, if $ g_{0} \leqslant h_{0} $, then one may take $ d = d_{1} - d_{2} $ and $ T' = T_{0} $.
\end{lemma} 

A proof of Lemma \ref{lemma2} can be achieved by following the same arguments as in \cite[Lemma 4.2]{Lebeau-Felipe}. For this reason, we omit the details. 

Let us now prove Theorem \ref{Theorem:Observability-Inequality}. Since the proof is similar to \cite[Theorem 1.1]{Lebeau-Felipe}, we only present the main ideas.

\begin{proof}[Proof of the Theorem \ref{Theorem:Observability-Inequality}]	
Let $ T \in (0, T_{0}) $ for some $ T_{0} > 0 $ to be chosen later, and let $ z_{0} \in H$. Consider the family of eigenfunctions $\{e_{j}\}_{j\in \mathbb{N}}$ associated with the Stokes eigenvalue problem \eqref{eig.pro}. For any $ \Lambda $, we define 
	$$
		H_{\Lambda} = Span\bigl\{e_{j} : \mu_{j} \leq \Lambda\bigr\}.
	$$

Let $z$ be the solution of \eqref{Adjoint.stk.eq} associated to $z_{0} \in H$. Decomposing $ z_{0} = z_{\Lambda,0} + z_{\Lambda,\perp,0}$, with $ z_{\Lambda,0} \in H_{\Lambda} $ and $ z_{\Lambda,\perp,0} \in H_{\Lambda}^{\perp} $, we split $z$ into its low and high frequency components, namely  $
z = z_{\Lambda} + z_{\Lambda,\perp}$, 
with $z_{\Lambda}(t) \in H_{\Lambda} $ and $z_{\Lambda,\perp}(t) \in H_{\Lambda}^{\perp}$, for every $t>0$.

It is not difficult to see that one has 
	\begin{equation}\label{des.0}
		\|z_{\Lambda,\perp}(\cdot,t)\|_{H} \leq e^{-\Lambda t}\|z_{\Lambda,\perp,0}\|_{H} \quad \forall t>0,
	\end{equation}
and that, for every $ M_{1} > 0$, 
	\begin{equation}\label{des.1}
		\|z(\cdot,\tau)\|_{H}^{2} \leq \frac{1}{M_{1}}e^{\frac{M_{1}}{\tau}}\int_{0}^{\tau}\|z(\cdot,t)\|_{L^{2}(\Omega)}^{2}dt, \quad \quad \forall\ \tau \in (0,T_{0}). 
	\end{equation}

From \eqref{des.1}, and the spectral inequality in Theorem \ref{spct-ineq-theorem}, we have that
\begin{equation}\label{des.2}
		\| z_{\Lambda}(\cdot,\tau) \|_{H}^{2} \leq\ \frac{M}{M_{1}}e^{\frac{M_{1}}{\tau} + K\sqrt{\Lambda}}\int_{0}^{\tau}\sum_{k=1}^{n-1}\|z_{k,\Lambda}(\cdot,t)\|_{L^{2}(\omega)}^{2}dt,
\end{equation}	
for every $\tau \in (0,T_{0})$ and any $M_1>0$, where $z_{k,\Lambda}$ is $k$-th component of $z_{\Lambda}$.

Next, we set the observation time $\tau = \epsilon T$, with $ \epsilon \in (0,1) $ small. Choosing $\sqrt{\Lambda} = \frac{1}{\epsilon T}$, estimate \eqref{des.2} then gives
	\begin{equation}\label{des.4}
		\|z_{\Lambda}(\cdot,T)\|_{H}^{2} \leq \frac{1}{4h(T)}\int_{(1-\epsilon)T}^{T}\sum_{k=1}^{n-1}\|z_{k,\Lambda}(\cdot,t)\|_{L^{2}(\omega)}^{2}dt,
	\end{equation}
with 
	\begin{equation*}
		h(T) := \frac{4 M_{1}}{M}e^{-\frac{M_{1}+K}{\epsilon T}}.
	\end{equation*}	
	
Using that $z_{k,\Lambda} = z_{k} - z_{k,\Lambda,\perp}$, it follows from \eqref{des.4} and energy estimate for $z_{k,\Lambda,\perp}$ in $[(1-\epsilon)T, T]$ that
	\begin{equation*}
	h(T)\|z_{\Lambda}(\cdot,T)\|_{H}^{2}\leq \frac{1}{2}\int_{(1-\epsilon)T}^{T}\sum_{k=1}^{n-1}\|z_{k}(\cdot,t)\|_{L^{2}(\omega)}^{2}dt + \frac{\epsilon T}{2}\sum_{k=1}^{n-1}\| z_{k,\Lambda,\perp}(\cdot,(1-\epsilon)T)\|_{H}^{2}.
\end{equation*}

From this last estimate, we get
	\begin{equation}\label{EstobsapproxABC}
		h(T)\|z(\cdot,T)\|_{H}^{2} \leq \int_{(1-\epsilon)T}^{T}\sum_{k=1}^{n-1}\|z_{k}(\cdot,t)\|_{L^{2}(\omega)}^{2}dt + \epsilon T\sum_{k=1}^{n-1}\| z_{k,\Lambda,\perp}(\cdot,(1-\epsilon)T)\|_{H}^{2} + 2h(T)\| z_{\Lambda,\perp}(\cdot,T)\|_{H}^{2}.
	\end{equation}

Therefore, combining \eqref{des.0} with \eqref{EstobsapproxABC}, and using the fact that $\|z_{\Lambda,\perp,0}\|_{H} \leq \|z_{0}\|_{H}$, we obtain 
	\begin{equation*}
		h(T)\|z(\cdot,T)\|_{H}^{2} - g(T)\|z_{0}\|_{H}^{2} \leq \int_{(1-\epsilon)T}^{T}\sum_{k=1}^{n-1}\|z_{k}(\cdot,t)\|_{L^{2}(\omega)}^{2}dt, 
	\end{equation*}	
for all $T \in (0,T_{0})$, with 
	\begin{equation*}
		g(T) := \bigl(T_{0} + 2h(T_{0})\bigr)e^{-\frac{2(1-\epsilon)}{\epsilon^{2}T}}.
	\end{equation*}	
	
From this last estimate and Lemma \ref{lemma2}, we conclude the proof of Theorem \ref{Theorem:Observability-Inequality}. 
\end{proof}

\newpage

\bigskip
\section{Further Comments and Open Problems}\label{Section:Applications}

The results obtained in this paper have several applications. For instance, they may be applied to investigate the null controllability of some nonlinear  versions of the Stokes system, null controllability on positive measurable sets, as well as the Bang-Bang property in time-optimal control problems for the Stokes system.

This work also opens up some directions for future research, including the extension of our results to more general linearized fluid models or the obtainment of the optimal cost of controllability for the three-dimensional Navier-Stokes system when a single scalar control acts on the velocity field.

\subsection*{Source Term Method}

In this work, we have shown that the cost of controllability for the Stokes system \eqref{stk.eq} satisfies
	$$
		C_{obs}(T) \leq C_{1}e^{C_{2}/T},
	$$
where $C_{1}$ and $C_{2}$ are positive constants. Using this estimate, it is possible to apply the so-called \textit{Source Term Method}, introduced by Liu, Takahashi, and Tucsnak in \cite{LiuTakahashiTucsnak}. This method provides a way to construct a weighted functional space of source terms for which the null controllability of the Stokes system with $n-1$ scalar controls still holds. Finally, it is possible to apply an inverse mapping theorem, or a fixed-point argument, to obtain local null controllability for some nonlinear perturbations of the Stokes system with $n-1$ scalar controls.

\medskip

\subsection*{Null controllability with bounded controls on sets of positive measure}

We can also study the $ L^{\infty}$-null controllability problem for the Stokes system when the control domain $\mathcal{O} \subset \Omega \times (0,T)$ is a measurable subset of positive measure. More precisely, we can ask the following question:
\begin{quote}
\textit{Let $T>0$. Given $u_{0} \in H$, does there exist a control $ f := (f_{1}, \dots, f_{n-1},0) \in L^{\infty}(\mathcal{O})^{n}$ such that the corresponding solution to \eqref{stk.eq} satisfies}
	\begin{equation*}
		u(\cdot,T) = 0 \quad \hbox{in}\ \Omega?
	\end{equation*}
\end{quote}

Adapting the arguments  in \cite{Souza-CS-Zhang}, it is not difficult to see that a positive answer to this question  is a consequence of the following result.

\begin{theorem}\label{obser}
Let $T>0$ and let $\mathcal{O} \subset \Omega \times (0,T)$ be a measurable subset of positive measure. There exists a positive constant $C_{obs}=C(n,\Omega,\mathcal{O},T)$ such that the observability inequality
	\begin{equation}\label{Obsinequality}
		\|z(\cdot,T)\|_{H} \leq C_{obs}\sum_{k=1}^{n-1}\iint_{\mathcal{O}}|z_{k}(x,t)|dxdt
	\end{equation}	
holds for all $z_{0} \in H$, where $z$ (with corresponding pressure) solves the associated adjoint system \eqref{Adjoint.stk.eq}.
\end{theorem}

In order to prove  Theorem \ref{obser}, we need to extend the spectral inequality given in \eqref{spectral-inequality-1} to the case of measurable sets of positive measure.
\begin{theorem}\label{4241}
Let  $\omega \subset \Omega$ be a measurable set with positive measure. Then, there exists a constant $C = C(n,\Omega,|\omega|) > 0$ such that 
	\begin{equation*}
		\left(\sum_{\mu_j \leq \Lambda}a_{j}^{2}\right)^{1/2}  
		\leq Ce^{C\sqrt{\Lambda}} \int_{\omega}\sum_{k=1}^{n-1}\biggl|\sum_{\mu_{j}\leqslant \Lambda}a_{j}e_{j,k}(x)\biggr|dx,
	\end{equation*}
for all $\Lambda > 0$ and any sequence of real numbers $(a_{j})_{j\geq1} \in \ell^{2}$.
\end{theorem}

A complete proof of Theorems \ref{obser} and \ref{4241} will be presented in a forthcoming paper. 

\begin{obs}
It is worth mentioning that Theorem \ref{4241} provides a complete solution to the open problem proposed in \cite[Remark~3.5]{lucero_kevin}.
\end{obs}

\medskip


\subsection*{Time Optimal Control for the Stokes System in the bi-dimensional case}

The results of this article can be applied to   give a positive answer to the following question   raised in~\cite[Remark~3.9]{Souza-CS-Zhang}:
\begin{quote}
	\textit{In the bi-dimensional case, can one get the stronger bang-bang property, for time optimal controls for the Stokes system?}
\end{quote}

To address this question, for a given $\omega \subset \Omega$ of positive measure and any $M>0$, we define the \textit{admissible control set} as
\begin{equation*}
	\mathcal{U}_{ad}^{M} := \Bigl\{(v_1,v_2) \in L^{\infty}(\omega\times[0,\infty))^2 : v_{2} \equiv 0 \hbox{ and } |v_{1}(x,t)| \leq M \hbox{\, a.e. in \,} \omega\times[0,\infty)\Bigr\}
\end{equation*}
and, for a  given initial state $u_{0} \in H$, the \textit{set of reachable state} from $u_0$ associated with controls in $\mathcal{U}_{ad}^{M}$  defined by
\begin{equation*}
	\mathcal{R}\left(u_{0},\mathcal{U}_{ad}^{M}\right) := \Bigl\{u(\cdot,\tau) : \tau > 0 \hbox{ and } u \hbox{ is the solution of \eqref{stk.eq} with } v \in \mathcal{U}_{ad}^{M} \Bigr\}.
\end{equation*}

Using Theorem \ref{obser}, we can show that $ 0 \in \mathcal{R}\left(u_0,\mathcal{U}_{ad}^{M}\right)$ and, therefore,  the following \textit{time-optimal control problem} is well-defined: 
\begin{quote}
\textit{Given $u_{0} \in H$ and $u_{f} \in \mathcal{R}\left(u_{0},\mathcal{U}_{ad}^{M}\right)$, find $ v^{\star} \in \mathcal{U}^{M}_{ad}$ such that the corresponding solution $u^{\star}$ of \eqref{stk.eq} satisfies}
	\begin{equation}\label{timeoptimalproblem0}
		u^{\star}(\tau^{\star}(u_{0},u_{f})) = u_{f},
	\end{equation}
\textit{where $\tau^{\star}(u_{0},u_{f})$ is the minimal time needed to steer the initial datum $u_{0}$ to the final state $u_{f}$ with controls in $\mathcal{U}_{ad}^{M}$, i.e.}
	\begin{equation}\label{timeoptimalproblem}
		\tau^{\star}(u_{0},u_{f}) = \min_{ v \in \mathcal{U}_{ad}^{M}}\left\{\tau : u(\cdot,\tau) = u_{f}\right\}.
	\end{equation}
\end{quote}

The following result shows that optimal controls satisfy a bang–bang property.

%
\begin{proposition}\label{unique-opt-ctrl}
Let $M>0$. For any $u_{0} \in  H$ and every $u_{f} \in \mathcal{R}\left(u_{0},\mathcal{U}_{ad}^{M}\right)$, the time optimal control problem \eqref{timeoptimalproblem0}-\eqref{timeoptimalproblem}  has a unique solution $ v^{\star}$ which satisfies a bang-bang property: 
$$
|v^{\star}_{1}(x,t)| = M \quad  \text{for a.e.} \ (x,t) \in \omega\times[0,\tau^{\star}(u_{0},u_{f})].
$$
\end{proposition} 

For the sake of brevity, the proof of Proposition \ref{unique-opt-ctrl} is omitted. It will be given in a forthcoming paper.

\section*{Cost of Controllability for Linearized fluid systems}



A natural question arising from our results concerns the optimal cost of null controllability for $n$ dimensional linearized fluid models of the form
	\begin{equation}\label{linea-oseen-control}
		u_{t} - \nu\Delta u + (a(x,t) \cdot \nabla)u + (u \cdot \nabla)b(x,t) + \nabla p = f1_{\omega}, 
	\end{equation}
where $a = a(x,t)$ and $b=b(x,t)$ are suitable vector fields and the control acts only  through $n-1$ components of the system.

When the control acts on all $n$ components of the system,  we refer to the recent works  \cite{Remi-Ludovick} and \cite{Remi-Takeo}. In the former,  the authors study a linearization of the Navier–Stokes operator around a laminar flow, under no-slip boundary conditions, and establish a spectral inequality for the associated Oseen operator. This yields upper bounds for the cost of null controllability of order $\mathcal{O}\bigl(e^{C/T^{1+\epsilon}}\bigr)$, for any $\epsilon > 0$. In the later, using precise estimates on the pressure, the authors obtain a Carleman estimate for the Oseen system on bounded domains, with weights analogous to those used for the heat equation, and use it to deduce that the cost of null controllability for the Oseen system  is indeed of order $\mathcal{O}\bigl(e^{C/T}\bigr)$. 

As far as we know, the optimal cost of null controllability for system \eqref{linea-oseen-control} in the case of control having $n-1$ components is still an open problem.

\section*{Cost of controllability for the 3D Navier-Stokes system with one control}

It is well known (see \cite{LionsZuazua}) that, in general, null controllability for the three-dimensional Stokes system cannot be achieved when the control has two vanishing components. Nevertheless, J.-M. Coron and P. Lissy, in \cite{Coron&Lissy}, combining the return method with the fictitious control method, obtained local null controllability for the three-dimensional Navier-Stokes system under the action of a single scalar control.

This naturally raises the question of determining the optimal cost of null controllability  for the three-dimensional Navier-Stokes system. We believe that combining the results of \cite{Remi-Takeo} with the arguments developed in \cite{Coron&Lissy} should lead to a controllability cost of order $\mathcal{O}\bigl(e^{C/T}\bigr)$. However, to the best of our knowledge, this remains an open problem.

%
%
%
\section*{Acknowledgments}
F. W. Chaves-Silva has been partially supported by CNPq-Bolsas de Produtividade em Pesquisa (project  \textit{``Estudos Avan\c{c}ados em Controle de Equa\c{c}\~oes em Derivadas Parciais: Fundamentos e Aplica\c{c}\~oes''}) and by  CNPq-Programa Conhecimento Brasil (project  \textit{``An\'alise e Controle de Fen\^omenos n\~ao-Lineares''}). D.~A.~Souza was partially supported by Grants PID2024-158206NB-I00 and CNS2024-154725 funded by MICINN (Spain). M.~G.~Ferreira-Silva was partially supported by a CAPES scholarship and CNPq Grant 200714/2022-8.
\appendix

\section{Proof of Proposition \ref{Prop:w2-estimate}}\label{Appendix:Proof-of-Prop-3.1}

In this section, we prove Proposition \ref{Prop:w2-estimate}, which was used in the proof of Theorem \ref{thrm.int.inq}.

For completeness, let us recall the definitions of $u_{\Lambda}$ and $v_{\Lambda}$: 
	$$
		u_{\Lambda}(s,x) = \sum_{\mu_{j} \leq \Lambda}a_{j}\frac{\sinh(s\sqrt{\mu_{j}})}{\sqrt{\mu_{j}}}e_{j}(x) \quad \text{and} \quad v_{\Lambda}(s,x) = \Delta_{x}u_{\Lambda}(s,x),
	$$
where $e_{j} = (e_{j,1}, e_{j,2}, \dots, e_{j,n})$ denote the eigenfunctions of the Stokes operator. During this section, the constants $C$ may vary from line to line.

In the following result, we estimate, for each $s \in (0,1)$, the $L^{2}$-norm of $u_{\Lambda}(s,\cdot)$ in the terms of the $L^{2}$-norm of $n-1$ components of $v_{\Lambda}(s,\cdot)$.
\begin{lemma}\label{Lemma:sum.to.n-1.components} 
Let $\widehat{v}_{\Lambda}$ be defined in \eqref{def:widehat(v)_Lambda}. There exists a constant $C > 0$, independent of $\Lambda$ and $s$, such that, for every $s \in (0,1)$, one has
	\begin{equation*}
		\sum_{\mu_{j}\leqslant \Lambda}|A_{j}(s)|^{2} \leq C\|\widehat{v}_{\Lambda}(s,\cdot)\|_{L^{2}(\Omega)^{n-1}}^{2},
	\end{equation*}
where we use the notation $A_{j}(s) = a_{j}\frac{\sinh(s\sqrt{\mu_{j}})}{\sqrt{\mu_{j}}}$.
\end{lemma}	
\begin{proof}
Fix $s \in (0,1)$. Since $v_{\Lambda}(s,\cdot) \in H^{2}(\Omega)^{n} \cap H_{0}^{1}(\Omega)^{n}$, one has the elliptic estimate 
	$$
	\|\widehat{u}_{\Lambda}(s,\cdot)\|_{L^{2}(\Omega)^{n-1}}^{2} 
	+ \Big\|\sum_{i=1}^{n-1}\partial_{x_{i}}u_{\Lambda,i}(s,\cdot)\Big\|_{L^{2}(\Omega)}^{2} 
	 \leq  C \|\widehat{v}_{\Lambda}(s,\cdot)\|_{L^{2}(\Omega)^{n-1}}^{2},
	$$
for some $C>0$.

The result then follows from the divergence-free condition and Poincaré’s inequality.
\end{proof}

The next result gives an estimate for the weight function $\varphi$ near the boundary.
\begin{lemma}\label{claim-1}
Let $\varphi$ be defined in \eqref{weigh-function-varphi(s,x)} and let $s_{*}>0$ be small. If $r_{0}>0$ is such that $\sqrt{r_{0}} < s_{*}$, then there exists $ \epsilon_{0} > 0 $ such that
	\begin{equation*}
		\varphi(s,r) \leq \varphi(s_{0},0) - \epsilon_{0} , \quad \forall\ (s,r) \in \left\{(s,r) : |s-s_{0}| > s_{*},\ r \leq r_{0}\right\}.
	\end{equation*}
\end{lemma}
\begin{proof}
For $r \leq r_{0}$ and $|s-s_{0}| > s_{*}$, we have $\varphi(s,r) < e^{D(r_{0} - s_{*}^{2})}$. 

Since $\varphi(s_{0},0) = 1$, it follows that
	\begin{equation*}
		\varphi(s,r) \leq \varphi(s_{0},0) - \left[1 - e^{D(r_{0} - s_{*}^{2})}\right].
	\end{equation*}

The conclusion follows by taking $\epsilon_{0} := 1 - e^{D(r_{0} - s_{*}^{2})}$. 
\end{proof}

\null

We now prove Proposition \ref{Prop:w2-estimate}, which establishes that the semiclassical $H^{1}$-norm of $w_{2}$ can be estimated in terms of the $n-1$-component functions $\widehat{w}$ and $\widehat{v}_{\Lambda}$. The proof relies on the two main estimates presented below as claims.
\begin{claim}\label{two-first-terms-estimates-H1sc-norm}
Let $w_{2}$ be defined \eqref{w0-splitted}, $N \in \mathbb{N}$, $0 < \rho < 1$ and $\epsilon_{0} > 0$. Then, there exists a constant $C>0$, independent of $\Lambda$ and $h$, such that for every $h \in (0,h_{1})$ one has
	\begin{equation*}
		\begin{split}
			\int_{0}^{1}\|w_{2}(s,\cdot)\|_{L^{2}(\partial\Omega)^{n}}^{2}ds & + \int_{0}^{1}h^{2}\|\nabla_{\partial\Omega}w_{2}(s,\cdot)\|_{L^{2}(\partial\Omega)^{n}}^{2}ds\\
			& \leq C\Lambda^{-2N}\Bigl(\|\widehat{w}\|_{L^{2}(\mathcal{Z})^{n-1}}^{2} + e^{\frac{2C_{0}}{h}}\left\|\widehat{v}_{\Lambda}\right\|_{H^{1}(\mathcal{Z})^{n-1}}^{2(1-\rho)}\left\|\partial_{s}\widehat{v}_{\Lambda}(0,\cdot) \right\|_{L^{2}(\omega)^{n-1}}^{2\rho}\Bigr).
		\end{split}
	\end{equation*}
\end{claim}

\begin{claim}\label{Claim:partial_s-estimate}
	Under the same assumptions as in Claim \ref{two-first-terms-estimates-H1sc-norm}, we have
	\begin{equation*}
		\begin{split}
			\int_{0}^{1}h^{2}\left\|\partial_{s}w_{2}(s,\cdot)\right\|_{L^{2}(\partial\Omega)^{n}}^{2}ds \leq C\Lambda^{-2N}\Bigl(& \|\widehat{w}\|_{L^{2}(\mathcal{Z})^{n-1}}^{2} + e^{\frac{2\left(\varphi_0(s_{0}) - \epsilon_{0}\right)}{h}}\|\widehat{v}_{\Lambda}\|_{L^{2}(\mathcal{Z})^{n-1}}^{2}\\
			& +\ e^{\frac{2C_{0}}{h}}\left\|\widehat{v}_{\Lambda}\right\|_{H^{1}(\mathcal{Z})^{n-1}}^{2(1-\rho)}\left\|\partial_{s}\widehat{v}_{\Lambda}(0,\cdot) \right\|_{L^{2}(\omega)^{n-1}}^{2\rho}\Bigr).
		\end{split}
	\end{equation*}
\end{claim}

The proof of Proposition \ref{Prop:w2-estimate} then follows by combining  Claim \ref{two-first-terms-estimates-H1sc-norm}, Claim \ref{Claim:partial_s-estimate}, and the definition of the $H_{sc}^{1}$-norm in the boundary.

We now proceed to the proofs of the two Claims.

\begin{proof}[Proof of Claim \ref{two-first-terms-estimates-H1sc-norm}]
Let 
	\begin{equation}\label{Phi(s)-def}
		\Phi_\Lambda(s) :=|\chi_{0}(s)|^{2}e^{\tfrac{2\varphi_{0}(s)}{h}}\sum_{\mu_{j}\leqslant \Lambda}|A_{j}(s)|^{2}.
	\end{equation}

From \cite[Lemma B.1]{Lebeau-Felipe}, for any $N \in \mathbb{N}$ and every $s \in (0,1)$, there exists $C_N >0$ such that 
	\begin{equation}\label{first-estimate-w2}
		\|w_{2}(s,\cdot)\|_{L^{2}(\partial\Omega)^{n}}^{2} + h^{2}\|\nabla_{y}w_{2}(s,\cdot)\|_{L^{2}(\partial\Omega)^{n}}^{2} \leq C_{N}\Lambda^{-2N}\Phi_\Lambda(s).
	\end{equation}

Thus, to prove the result, we just need to estimate $\Phi_\Lambda(s)$. 

First, we notice that
	\begin{equation*}
		\int_{\Omega}|\widehat{w}(s,\cdot)|^{2}dx \geq |\chi_{0}(s)|^{2}\int_{\Omega}\left|e^{\tfrac{\varphi(s,r)}{h}}\widehat{v}_{\Lambda}(s,\cdot)\right|^{2}dx - |\chi_{0}(s)|^{2}\int_{\partial\Omega}\int_{r \geq \frac{r_{0}}{2}}\left|e^{\tfrac{\varphi(s,r)}{h}}\widehat{v}_{\Lambda}(s,y,r)\right|^{2}drdy.
	\end{equation*}
Then, since $\varphi(s,r) \geqslant \varphi_{0}(s)$ for every $ r \geq 0$, from \eqref{Phi(s)-def} and Lemma \ref{Lemma:sum.to.n-1.components}, one has 
	\begin{equation*}
		\Phi_\Lambda(s) \leq C |\chi_{0}(s)|^{2}\int_{\Omega}\left|e^{\tfrac{\varphi(s,r)}{h}}\widehat{v}_{\Lambda}(s,\cdot)\right|^{2}dx,
	\end{equation*}
for some $C>0$. 

Hence, there exists $C>0$ such that 
	\begin{equation}\label{integral-estimate-04}
		\int_{0}^{1}\Phi_\Lambda(s)ds \leq C\left(\|\widehat{w}\|_{L^{2}(\mathcal{Z})^{n-1}}^{2} + \int_{0}^{1}\int_{\partial\Omega}\int_{r \geq \frac{r_{0}}{2}}\left|\chi_{0}(s)e^{\tfrac{\varphi(s,r)}{h}}\widehat{v}_{\Lambda}(s,y,r)\right|^{2}drdyds\right).
	\end{equation}

Using that $\supp(\chi_0) \subset \left(\frac{1}{9},\frac{9}{10}\right)$, from Proposition \ref{Prop-interpolation-Leb-Fel} there exist constants $C, C_0>0$ and $ \rho \in (0,1)$ such that 
	\begin{align} \label{integral-estimate-05}
		\int_{0}^{1}\int_{\partial\Omega}\int_{r \geq \frac{r_{0}}{2}}\left|\chi_{0}(s)e^{\tfrac{\varphi(s,r)}{h}}\widehat{v}_{\Lambda}(s,y,r)\right|^{2}drdyds & \leq  e^{\frac{2C_{0}}{h}}\left\|\widehat{v}_{\Lambda}\right\|_{L^{2}\left(\mathcal{W}_{r_{0}}\right)^{n-1}}^{2} \\
		& \leq Ce^{\frac{2C_{0}}{h}}\left\|\widehat{v}_{\Lambda}\right\|_{H^{1}(\mathcal{Z})^{n-1}}^{2(1-\rho)}\left\|\partial_{s}\widehat{v}_{\Lambda}(0,\cdot) \right\|_{L^{2}(\omega)^{n-1}}^{2\rho}. \nonumber	\end{align}

The result follows by combining \eqref{first-estimate-w2}, \eqref{integral-estimate-04} and \eqref{integral-estimate-05}.
\end{proof}



We now prove Claim \ref{Claim:partial_s-estimate}.  To do this, we consider the function $\chi_0$ given in \eqref{Def:characteristic-functions} and introduce a new  cut-off function $\widetilde{\chi}_0 \in \mathcal{C}_{c}^{\infty}\bigl( (s_{0}-2s_{*}, s_{0}+2s_{*}) \bigr) $, such that $ 0 \leq \widetilde{\chi}_{0} \leq 1 $ and $ \widetilde{\chi}_{0} = 1 $ on $\operatorname{supp}(\chi_{0}) $ and  an auxiliary function $\widetilde{w}$ given by
	\begin{equation*}
		\widetilde{w}(s,y,r) := e^{\frac{\varphi(s,r)}{h}}\widetilde{\chi}_{0}(s)\chi_{1}(r)\widehat{v}_{\Lambda}(s,y,r).
	\end{equation*} 

The following lemma gives an estimate for the $L^{2}$-norm of $\widetilde{w}$ in terms of $\widehat{w}$ and $\widehat{v}_{\Lambda}$.
\begin{lemma}\label{Lemma:estimate.for.w-tilde}
There exists a constant $\epsilon_{0} > 0$, depending only on $r_{0}, s_{*}$ and $D$, such that for every $h\in (0,h_{1})$, the following estimate holds:  
	\begin{equation*}
		\|\widetilde{w}\|_{L^{2}(\mathcal{Z})^{n-1}}^{2} \leq \|\widehat{w}\|_{L^{2}(\mathcal{Z})^{n-1}}^{2} + e^{\frac{2\left(\varphi_0(s_{0}) - \epsilon_{0}\right)}{h}}\|\widehat{v}_{\Lambda}\|_{L^{2}(\mathcal{Z})^{n-1}}^{2}.
	\end{equation*} 
\end{lemma}
\begin{proof}
By construction, the following estimate holds:
	\begin{equation}\label{cut-off-functios.inequality}
		\widetilde\chi_{0}(s) \leq \chi_{0} (s) + 1_{\left\{s \in [0,1] : \frac{3s_{*}}{2}< |s - s_{0}|< 2s_{*}\right\}}(s), \quad \forall\ s \in [0,1].
	\end{equation}
Hence, since $\widehat{w} = e^{\frac{\varphi}{h}}\chi_{0}\chi_{1}\widehat{v}_{\Lambda}$, Lemma \ref{Lemma:estimate.for.w-tilde} follows using \eqref{cut-off-functios.inequality} and Lemma \ref{claim-1}.
\end{proof}

\begin{proof}[Proof of Claim \ref{Claim:partial_s-estimate}]
Since
	\begin{equation*}
		\partial_{s}w_{2}(s,y) = (\mathrm{Id} - \mathcal{B}_{y})\left[\sum_{\mu_{j} \leqslant \Lambda}\partial_{s}\left(A_{j}(s)\chi_{0}(s)e^{\tfrac{\varphi_{0}(s)}{h}}\right)\Delta_{\partial\Omega} e_{j}(y,0)\right],
	\end{equation*}
applying \cite[Lemma B.1]{Lebeau-Felipe} with $M=0$, for all $ N \in \mathbb{N}$ there exists a positive constant $ C > 0 $ such that
	\begin{equation*}
		\left\|\partial_{s}w_{2}(s,\cdot)\right\|_{L^{2}(\partial\Omega)^{n}}^{2} \leq C\Lambda^{-2N}\sum_{\mu_{j} \leqslant \Lambda}\left|\partial_{s}\left(A_{j}(s)\chi_{0}(s)e^{\tfrac{\varphi_{0}(s)}{h}}\right)\right|^{2}.
	\end{equation*}

Since $h\sqrt{\Lambda} = \delta \ll 1$, by the definition of $\widetilde{\chi}_{0}$ and the fact that $|\partial_{s}\varphi_{0}(s)|$ is bounded independently of $s$, one has
	\begin{equation}\label{partial.s(w2)-est-02}
	\left\|h\partial_{s}w_{2}(s,\cdot)\right\|_{L^{2}(\partial\Omega)^{n}}^{2} \leq C\Lambda^{-2N}\widetilde\Phi_\Lambda(s),
\end{equation}
with 
$$
\widetilde\Phi_\Lambda(s) = \sum_{\mu_{j} \leq \Lambda}\left|\widetilde{\chi}_{0}^{2}(s)\right|^{2}e^{\tfrac{2\varphi_{0}(s)}{h}}|A_{j}(s)|^{2}.
$$	

Arguing as in \eqref{integral-estimate-05}, there exist positive constants $C$ and $C_0$ such that
\begin{equation}\label{integral-estimate-06}
	\int_{0}^{1}\widetilde\Phi_\Lambda(s)ds \leq C\left(\|\widetilde{w}\|_{L^{2}(\mathcal{Z})^{n-1}}^{2} + e^{\frac{2C_{0}}{h}}\left\|\widehat{v}_{\Lambda}\right\|_{H^{1}(\mathcal{Z})^{n-1}}^{2(1-\rho)}\left\|\partial_{s}\widehat{v}_{\Lambda}(0,\cdot)\right\|_{L^{2}(\omega)^{n-1}}^{2\rho}\right).
\end{equation}

 It is not difficult to see that Claim \ref{Claim:partial_s-estimate} is then a consequence of \eqref{partial.s(w2)-est-02}, \eqref{integral-estimate-06} and Lemma \ref{Lemma:estimate.for.w-tilde}.
\end{proof}



\section{Proof of  Proposition \ref{Proposition:(1-Bs)w_1--estimate}}\label{Appendix:Proof-Claim.3.2--and--Prop.3.3}

Now we prove Proposition \ref{Proposition:(1-Bs)w_1--estimate}. This result allows us to estimate the high-frequency part in $\sigma$ of all components of $w_{1}$ at the boundary in terms of a right-hand side involving only $n-1$ components in the interior.

\begin{proof}[Proof of Proposition \ref{Proposition:(1-Bs)w_1--estimate}]
Let us set $ \nu_{j} := h^{2}\mu_{j} $, which satisfies $ \sqrt{\nu_{j}} \leq h\sqrt{\Lambda} = \delta $. Then, we write
	\begin{equation*}
		w_{1}(s,y) = \sum_{\mu_{j}\leq\Lambda}a_{j}\chi_{0}(s)\left(\frac{e^{h^{-1}\left[\varphi_{0}(s) + s\sqrt{\nu_{j}}\right]} - e^{h^{-1}\left[\varphi_{0}(s) - s\sqrt{\nu_{j}}\right]}}{2} \right)\mathcal{B}_{j}(y),
	\end{equation*}	
with $ \mathcal{B}_{j} $ given by
	$$
		\mathcal{B}_{j}(y) := \frac{\mathcal{B}_{y}\Bigl(\Delta e_{j}\bigl|_{\partial\Omega}\Bigr)(y,0)}{\sqrt{\mu_{j}}}.
	$$ 
 
From \eqref{pseudo-differential-operator-in-s-var} and making the change of variable $s \mapsto s_{0} + t $, we obtain
	\begin{equation*}
		\mathfrak{B}_{s}w_{1}(t) = \frac{1}{2\pi h}\int e^{\frac{it\sigma}{h}}b_{1}(\sigma)\left[\int e^{-\frac{i t'\sigma}{h}}b_{0}(t')w_{1}(s_{0} +  t',y)d t'\right]d\sigma.
	\end{equation*}

By setting $s' = s_{0} +  t'$, with $t' \in \mathbb{R}$, we may write 
	$$
		\varphi_{0}(s') = \varphi_{0}(s_{0}) - D(s'-s_{0})^{2} + \mathcal{O}\bigl((s' - s_{0})^{4}\bigr),
	$$
where $D>0$, and hence, the real-analytic function defined by 
	$$
		\theta_{0}(t') := \varphi_{0}(s')- \varphi_{0}(s_{0})
	$$ 
can be rewritten as $\theta_{0}(t') = D t'^{2} + \mathcal{O}(t'^{4})$.
	
By defining the phase function 
	\begin{equation}\label{phase-function-rho_j}
		\rho_{j}^{\pm}(\sigma, t') := -\sigma  t' + i\left[\theta_{0}( t') \pm (s_{0} +  t')\sqrt{\nu_{j}}\right],
	\end{equation}
one has
	\begin{equation*}
		(1 - \mathfrak{B}_{s})w_{1}(t) = \frac{e^{\frac{\varphi_{0}(s_{0})}{h}}}{4\pi h}\sum_{\mu_{j}\leq\Lambda}a_{j}\mathcal{B}_{j}(y)\int e^{\frac{it\sigma}{h}}(1 - b_{1}(\sigma))\left[\int e^{h^{-1}i\rho_{j}^{\pm}(\sigma, t')}\widetilde{\chi}( t')d t'\right]d\sigma,
	\end{equation*}
with
	\begin{equation}\label{hat-chi-def}
		\widetilde{\chi}(t') := b_{0}(t')\chi_{0}(s_{0} + t').
	\end{equation}

Since $1 - b_{1}$ is compactly supported in the set $\left\{\sigma : |\sigma| \geq \frac{\alpha_{1}}{2}\right\}$, it follows that
	\begin{equation*}
		\|(1 - \mathfrak{B}_{s})w_{1}\|_{H_{sc}^{1}(\{r=0\})^{n}} \leq \frac{e^{\frac{\varphi_{0}(s_{0})}{h}}}{4 \pi h}\sum_{\mu_{j}\leq\Lambda}|a_{j}|\bigl\|\mathcal{B}_{j}\bigr\|_{H_{sc}^{1}(\{r=0\})^{n}}\int_{|\sigma| \geq \frac{\alpha_{1}}{2}}\langle \sigma \rangle\left|\int e^{h^{-1}i\rho_{j}^{\pm}(\sigma, t')}\widetilde{\chi}( t')d t'\right|d\sigma,		
	\end{equation*}	
where $\langle \sigma \rangle = (1 + \sigma^{2})^{\frac{1}{2}}$.
	
By the definition of $\mathcal{B}_{j}$, and trace theorem, there exist positive constants $C$ and $m'$ such that
	\begin{equation*}
		\bigl\|\mathcal{B}_{j}\bigr\|_{H^{1}(\{r=0\})^{n}} \leq C\Lambda^{m'}.
	\end{equation*}
Therefore, from Lemma \ref{Lemma:sum.to.n-1.components} we obtain
	\begin{equation}\label{(1-Oph)w1-estimate-02}
		\begin{split}
			\bigl\|(1 - \mathfrak{B}_{s})&w_{1}\bigr\|_{H_{sc}^{1}(\{r=0\})^{n}}\\
			&  \leq \frac{C\Lambda^{m}e^{\frac{\varphi_{0}(s_{0})}{h}}}{h}\left(\sup_{j ; \nu_{j}\leq h^{2}\Lambda}\int_{|\sigma| \geq \frac{\alpha_{1}}{2}}\langle \sigma \rangle\left|\int e^{h^{-1}i\rho_{j}^{\pm}(\sigma, t')}\widetilde{\chi}( t')d t'\right|d\sigma\right)\|\widehat{v}_{\Lambda}\|_{H^{1}(\mathcal{Z})^{n-1}},		
		\end{split}
	\end{equation}
where $m := m' + \frac{n}{4} + \frac{1}{2}$ and $C = C(|\Omega|,n) >0$.

Proposition \ref{Proposition:(1-Bs)w_1--estimate} then follows by combining \eqref{(1-Oph)w1-estimate-02}, the fact that $h\sqrt{\Lambda} = \delta \leq 1$, and the following Lemma, whose proof can be found in \cite[Lemma 2.6]{Lebeau-Felipe}.
\begin{lemma}\label{lemma-for-estimate-the-phase-integral}
Let $\rho_{j}^{\pm}$ be the phase function defined in \eqref{phase-function-rho_j}, and let $\widetilde{\chi}$ be as in \eqref{hat-chi-def}. Then, there exist constants $ c_{0} > 0 $ and $ \delta_{0} > 0$ such that, for all $ 0 < \delta \leq \delta_{0} $ and all $ k \in \mathbb{N} $, there exists a constant $ C_{k} > 0 $ satisfying the estimate
	\begin{equation*}
		\left| \int e^{h^{-1}i\rho_{j}^{\pm}(\sigma, t')}\widetilde{\chi}( t')d t' \right|\ \leq\ C_{k}\langle \sigma \rangle^{-k}e^{\tfrac{-c_{0}}{h}}
	\end{equation*}
	uniformly for $ |\sigma| \geq \tfrac{\alpha_1}{2} $ and $ \sqrt{\nu_{j}} \leq \delta $.
\end{lemma} 
\end{proof}


\bibliographystyle{unsrt}







%

\end{document}